\documentclass[11pt]{amsart}

\usepackage[margin=1in]{geometry}
\usepackage{amsthm}
\usepackage{amsmath}
\usepackage{amssymb}
\usepackage{mathrsfs}
\usepackage{amsfonts}
\usepackage{graphicx}
\usepackage{pgf,tikz}
\usepackage{color}
\usepackage{enumerate}
\usepackage{multicol}
\usepackage{graphpap}
\usepackage{tikz-cd}

\theoremstyle{plain}
\newtheorem{theorem}{Theorem}[section]
\newtheorem{proposition}[theorem]{Proposition} 
\newtheorem{lemma}[theorem]{Lemma}

\newtheorem{corollary}[theorem]{Corollary}
\newtheorem{conjecture}[theorem]{Conjecture}

\newtheorem{eg}[theorem]{Example}
\theoremstyle{definition}
\newtheorem{definition}[theorem]{Definition}
\newtheorem{remark}[theorem]{Remark}
\newtheorem{question}[theorem]{Question}

\newcommand{\N}{\mathbb N}
\newcommand{\R}{\mathbb R}
\newcommand{\Z}{\mathbb Z}

\newcommand{\eps}{\varepsilon}

\newcommand{\bfu}{\mathbf{u}}
\newcommand{\bfx}{\mathbf{x}}
\newcommand{\bfy}{\mathbf{y}}

\begin{document}
	
\title{Minimal Bounded Speedups of Toeplitz Flows}
\author[L.\ Alvin]{Lori Alvin}
\address{Department of Mathematics, Furman University, Greenville, SC 29613, USA}
\email{lori.alvin@furman.edu}

\author[S.\ Radinger]{Silvia Radinger}
\address{Faculty of Mathematics, University of Vienna, Vienna, Austria}
\email{silvia.radinger@univie.ac.at}


\begin{abstract}  
	We investigate minimal bounded speedups of Toeplitz flows. We demonstrate that the minimal bounded speedup of a Toeplitz flow need not be another Toeplitz flow and describe techniques for determining whether the resulting speedup is Toeplitz. We also provide sufficient conditions to guarantee that the minimal bounded speedup will be a Toeplitz flow.
\end{abstract}

\maketitle

\section{Introduction}\label{Sec:Intro}

Given a minimal Cantor system $(X,T)$, a {\em topological speedup} of $(X,T)$ is a dynamical system $(X,S)$ where $S$ is a homeomorphism such that $S(x) =T^{p(x)}(x)$ for some function $p:X\to \N$. We call the function $p(x)$ the {\em jump function} for the speedup. Throughout this paper we assume that $p(x)$ is continuous (equivalently bounded by \cite[Proposition 2.2]{AAO}) and that the resulting Cantor system $(X,S)$ is minimal. We are interested in investigating which properties of the original Cantor system are preserved under minimal bounded speedups. In particular, this paper focuses on speedups of Toeplitz flows. 

There is a rich history of studying speedups of dynamical systems, particularly in the measure-theoretical setting; see, for example, \cite{AOW,Neveu1,Neveu2}. More recently, there has been a focus on topological speedups. In \cite{AO}, the authors characterize when one Cantor minimal system is the speedup of another and provide topological analogues to the measure-theoretical results of \cite{AOW}. In \cite{AAO}, minimal bounded speedups of odometers and substitutions are investigated. It is shown that the minimal bounded speedup of an odometer is a conjugate odometer, whereas the minimal bounded speedup of a primitive substitution is always a different primitive substitution. In \cite{ADL}, the authors use Bratteli diagrams to investigate minimal bounded speedups of primitive substitutions. 

Because it is known that minimal bounded speedups of odometers are conjugate odometers, and Toeplitz flows are symbolic, minimal, almost one-to-one extensions of odometers, it was conjectured in \cite{ADL} that the minimal bounded speedup of a Toeplitz flow is a Toeplitz flow. We demonstrate that this conjecture is false. We investigate minimal bounded speedups of Toeplitz flows and provide sufficient conditions for a speedup to be Toeplitz.

In Section~\ref{sec:Prelim} we present background terminology and foundational results about substitution shifts, Toeplitz flows, minimal Cantor systems, and topological speedups of minimal Cantor systems. Section~\ref{sec: Not Toeplitz} demonstrates the existence of minimal bounded speedups of Toeplitz flows that are not themselves Toeplitz flows. In fact, we show in Theorem~\ref{Theorem: no speedup with c=2 is Toeplitz} that no minimal bounded speedup of the Toeplitz flow $(X_\theta,\sigma)$ generated by the substitution $\theta(a) = ab$ and $\theta(b) = aa$ with orbit number $c=2$ will be a Toeplitz flow.

We demonstrate in Section~\ref{Sec: Determining if Toeplitz} that a constant speedup $(X,T^c)$ of a Toeplitz flow $(X,T)$ is a Toeplitz flow if and only if $\gcd(c,p_k)=1$ for all $k$ where $(p_k)$ is the period structure of the original Toeplitz flow. Additionally, in this case $(X,T)$ and $(X,T^c)$ will have the same underlying odometer. The section concludes by presenting a method for determining whether a given speedup of a Toeplitz flow is Toeplitz. Strategies are provided for both the case where $(X,T)$ is generated by a constant length left proper substitution and the more general Toeplitz flow setting.

The final section of the paper investigates sufficient conditions for the minimal bounded speedup of a Toeplitz flow to be Toeplitz. Namely, if there exists some level of the Kakutani-Rokhlin partition for $(X,T)$ such that every tower has the same orbit permutation and each orbit has the same number of levels in each tower under the jump function $p$, then the speedup $(X,S)$ will be a Toeplitz flow (Theorem~\ref{Theorem: Sufficient Same Permutation and Orbits}). We demonstrate that this can happen such that $(X,S)$ is conjugate to $(X,T^c)$ and such that $(X,S)$ is not conjugate to $(X,T^c)$, as $(X,T^c)$ need not be minimal. We conjecture that our sufficient conditions in Theorem~\ref{Theorem: Sufficient Same Permutation and Orbits} are actually necessary for $(X,S)$ to be Toeplitz because of the strong relationship between Toeplitz flows and odometers.

\section{Preliminaries}\label{sec:Prelim}

\subsection{Substitution Subshifts} Let $\mathcal{A}$ be a finite {\em alphabet}. A {\em substitution} over $\mathcal{A}$ is a map $\theta: \mathcal{A} \to \mathcal{A}^*$, where $\mathcal{A}^*$ denotes the set of nonempty finite words with letters from $\mathcal{A}$. We extend $\theta$ to map from $\mathcal{A}^*$ to $\mathcal{A}^*$ by concatenation such that $\theta(ab) = \theta(a)\theta(b)$; we then extend to maps from $\mathcal{A}^\N$ to $\mathcal{A}^\N$ (similarly from $\mathcal{A}^\Z$ to $\mathcal{A}^\Z$) in the natural way. A substitution is {\em constant length} if $|\theta(a)| = |\theta(b)|$ for all $a,b\in \mathcal{A}$. We say a substitution is {\em primitive} if there exists some $n\in { \N = \{1,2,3\dots\}}$ such that for every $a\in \mathcal{A}$, $\theta^n(a)$ contains every letter of $\mathcal{A}$. A substitution is {\em left proper} if there is some $a\in \mathcal{A}$ such that $\theta(b)$ starts with $a$ for every $b\in \mathcal{A}$; if we also have some $c\in \mathcal{A}$ such that $\theta(b)$ ends with $c$ for every $b\in \mathcal{A}$, then we say the substitution is {\em proper}.

Without loss of generality (by considering $\theta^n$ if necessary), we assume $\theta$ is primitive if and only if there is a letter $a\in \mathcal{A}$ such that $\theta(a)$ begins with $a$, $\lim_{n\to \infty}|\theta^n(b)| = \infty$ for all $b\in \mathcal{A}$, and $\lim_{n\to\infty}\theta^n(a)$ is a fixed point under the substitution that contains every letter in $\mathcal{A}$. 
We define the shift space $X_\theta \subset \mathcal{A}^\Z$ to be the set of all $x\in \mathcal{A}^\Z$ such that whenever $i,j\in \Z$ with $i<j$, then $x_{[i,j]}$ is a subword of $\theta^k(a)$ for some $k\in \N$ and $a\in \mathcal{A}$.   We refer to $(X_\theta, \sigma)$ as the {\em substitution subshift} and emphasize that $\theta$ is primitive if and only if $(X_\theta,\sigma)$ is minimal.

{ We generalize these ideas to sequences of substitutions. Let $(\mathcal{A}_n)_n$ be a sequence of non-empty finite sets and $\mathcal{A}_n^\ast$ be the set of all finite non-empty words over the alphabet $\mathcal{A}_n$.
Take a sequence $(\theta_n)_{n \geq 2}$, where $\theta_n : \mathcal{A}_n \rightarrow \mathcal{A}_{n-1}^\ast$ is a substitutions from  $\mathcal{A}_n$ to $\mathcal{A}_{n-1}$for all $n$. Then $(X, \sigma)$ is called an \textit{$S$-adic subshift} if $X \subset \mathcal{A}_1^{\Z}$ is a set of sequences $x = (x_i)$ such that for all $i,j \in \Z$ the word $x_i x_{i+1}\ldots x_j$ appears in the word $\theta_2 \theta_3 \ldots \theta_n(a_n)$ for some $n\in \N$. The sequence $(\theta_n)_n$ is the \textit{$S$-adic representation} of $(X, \sigma)$. An $S$-adic representation $(\theta_n)_n$ is called primitive (or left-proper) if for every $n$ there exists an $i \geq 0$ such that $\theta_n \circ \cdots \circ \theta_{n+i}$ is primitive (or left proper).

}

\subsection{Toeplitz Flows}

A sequence $x\in \mathcal{A}^{\N}$ (or $\mathcal{A}^\Z$) is called a {\em Toeplitz sequence} if for every $i\in \N$ there exists a $p_i\in \N$ such that $x_i = x_{i+kp_i}$ for all $k\in \N$ (or $\Z$).  The minimal dynamical system given by $(X,\sigma)$  where $X = \overline{\{\sigma^n(x): n\geq 0\}}$ is called a {\em Toeplitz flow} or {\em Toeplitz shift}.

Let $(X,\sigma)$ be a Toeplitz flow with $X\subset \mathcal{A}^\Z$. Given $x\in X$, $p\in \N$, and $a\in \mathcal{A}$, we define the following sets.
\begin{itemize}
    \item $\text{Per}_p(x,a)= \{k\in \Z \mid x_{k'} = a \quad \forall k'\equiv k\mod p\}$.
    \item $\text{Per}_p(x) = \bigcup_{a\in \mathcal{A}}\text{Per}_p(x,a)$.
    \item $\text{Aper}(x) = \Z\setminus \bigcup_{p\in \N}\text{Per}_p(x)$.
\end{itemize}
A sequence $x\in \mathcal{A}^\Z$ is Toeplitz if and only if $\text{Aper}(x) = \emptyset$. Additionally, every Toeplitz sequence $x$ is such that for every $n\in \N$, there is a $p\in \N$ such that $\{0,1,2,3,\ldots, n-1\}\subseteq \text{Per}_p(x)$.

For any $p\in \N$ and $x\in \mathcal{A}^{\Z}$, the {\em p-skeleton} of $x$ is the sequence obtained from $x$ by replacing $x_i$ by a `hole' for all $i\notin \text{Per}_p(x)$.  That is, the p-skeleton of $x$ is the part of $x$ which is periodic with period $p$.
We call $p$ an {\em essential period} of a sequence $x$ if  the $p$-skeleton of $x$ is not periodic with smaller period.
\begin{definition}
    A {\em period structure} for an aperiodic Toeplitz sequence $x$ is an increasing sequence $\{p_i\}_{i\in \N}$ of positive integers satisfying the following properties.
  \begin{enumerate}
     \item For all $i\in \N$, $p_i$ is an essential period for $x$;
     \item $p_i \mid p_{i+1}$;
     \item $\bigcup_{i\geq 1} \text{Per}_{p_i}(x) = \Z$.
  \end{enumerate}
\end{definition}

There are many equivalent characterizations of Toeplitz sequences. We note a few of them here which will be used throughout this paper. The following theorem is stated in terms of one-sided Toeplitz sequences but can easily be generalized to two-sided sequences (see, for example, \cite{Durand_Perrin}). In fact, if $x\in \mathcal{A}^\Z$ is Toeplitz, then so is $x' = x_0x_1x_2\cdots\in \mathcal{A}^\N$. We state this one-sided version to aid in quickly identifying our KR-partitions.

\begin{theorem} \cite[Theorem 4.90]{Bruin_Book}\label{Theorem: S-adic Toeplitz Representation}
    The one-sided sequence $x\in \mathcal{A}^\N$ is Toeplitz if and only if there is a sequence of constant length substitutions $\theta_i: \mathcal{A}_i\to \mathcal{A}_{i-1}$ on finite alphabets $\mathcal{A}_i$ with $\mathcal{A} = \mathcal{A}_0$ such that {  each $\theta_i$ is left-proper}, and $x = \lim_{i\to\infty} \theta_1\circ\theta_2\circ \cdots \circ \theta_i(a)$, where $a\in \mathcal{A}_i$ is arbitrarily chosen.
\end{theorem} Observe that by using this characterization for the Toeplitz flow, we obtain a period structure $(p_k)$ where $p_k = |\theta_1\circ\theta_2\circ\cdots\circ\theta_k|$ for all $k\in \N$. 
We note that by allowing  $\theta_i= \theta$ for every $i\in \N$, we may obtain a system which is both a substitution shift and a Toeplitz flow. In this case, the underlying period structure is such that $p_k = |\theta|^k$.

\begin{corollary}
If $(X_\theta,\sigma)$ is generated by a constant length, left proper substitution $\theta$, then $(X_\theta,\sigma)$ is a Toeplitz flow.
\end{corollary}

We now define the $\alpha$-adic odometer, as there is a close relationship between odometers and Toeplitz flows.

\begin{definition}
    Let $\alpha = (\alpha_1,\alpha_2,
    \alpha_3,\ldots )$ be a sequence of integers with each $\alpha_i\geq 2$. Denote by $X_\alpha$ the set of all infinite sequences of points $\textbf{x} = (x_1,x_2,x_3,\ldots)$ such that $0\leq x_i \leq \alpha_{i}-1$ for each $i\in \N$.
    Addition on $X_\alpha$ is defined by $$(x_1,x_2,\ldots) + (y_1,y_2,\ldots) = (z_1,z_2,\ldots)$$ where $z_1 = (x_1+y_1) \mod \alpha_1$ and for each $i\geq 2$, $z_i = x_i+y_i+r_{i-1} \mod \alpha_i$ where $r_{i-1} = 0$ if $x_{i-1}+y_{i-1}+r_{i-2} < \alpha_{i-1}$ and $r_{i-1}=1$ otherwise (we set $r_0= 0$).
    We define $T_\alpha: X_\alpha \to X_\alpha$ by $$T_\alpha((x_1,x_2,x_3\ldots)) = (x_1,x_2,x_3, \ldots) + (1,0,0,\ldots).$$ The dynamical system $(X_\alpha, T_\alpha)$ is called the {\em $\alpha$-adic odometer}.
\end{definition}

\begin{theorem}\cite{Downarowicz}
    Toeplitz flows can be characterized up to topological conjugacy as topological dynamical systems that are symbolic, minimal, almost one-to-one extensions of odometers. Given the period structure $(p_k)$ of a Toeplitz point $x\in X$, the underlying odometer is $(X_\alpha,T_\alpha)$ where $\displaystyle \alpha = \left(p_1, \frac{p_2}{p_1},\frac{p_3}{p_2}, \ldots, \frac{p_i}{p_{i-1}}, \ldots \right)$.
\end{theorem}

\subsection{Kakutani-Rokhlin Partitions}
We now introduce Kakutani-Rokhlin partitions, which will be helpful in the study of topological speedups.

\begin{definition}
    Given a minimal Cantor system $(X,T)$, a {\em Kakutani-Rokhlin partition (KR-partition)} for $(X,T)$ is a partition $$\mathcal{P} = \bigcup_{i=1}^N \{T^j(B_i): 0\leq j\leq h_i-1\}$$ of $X$ into pairwise disjoint clopen sets that cover $X$. We call the set $B = \bigcup_{i=1}^N B_i$ the {\em base} of the KR-partition, { the sets $T^j(B_i)$ the {\em floors}} and the integers $h_i$ the {\em heights} of each {\em tower} $A_i = \{T^j(B_i): 0\leq j \leq h_i - 1\}$.  Further, for all $1\leq i \leq N$ we have $T^{h_i}(B_i) \subset B$.
\end{definition}

Every continuous minimal Cantor system $(X,T)$ has a sequence of KR-partitions { 
$$
\mathcal{P}(k)=\bigcup_{i=1}^{N_k}\{T^j B_i^k: 0\leq j\leq h_i^k - 1 \}
$$}
satisfying the following properties \cite{Herman_Putnam_Skau}.
{ \begin{enumerate}
    \item The sequence of bases is nested: $B^{k+1} \subset B^{k}$.
    \item For all $k\in \N$, $\mathcal{P}(k+1)$ refines $\mathcal{P}(k)$.
    \item $\bigcap_{k\in \N} B^{k}$ is a single point $x\in X$.
    \item The partitions $\mathcal{P}(k)$ form a basis for the topology of $X$.
    \item For all $k\in \N$, $i\leq N(k+1)$ and $i'\leq N(k)$, there exists $0\leq j\leq h_i^{k+1}-1$ such that $T^j(B_i^{k+1})\subset B_{i'}^{k}$.
    \item $B^{k+1}\subset B_1^{k}$ for all $k\in \N$.
\end{enumerate}
For ease, we denote each tower by $A_i^k = \{T^jB_i^k: 0\leq j\leq h_i^k-1\}$.}

We now explicitly define the KR-partitions $\mathcal{P}(k)$ for Toeplitz flows and substitution systems.

\begin{definition}\label{Definition: Toeplitz KR partitions}
    Suppose $(X,T)$ is a Toeplitz flow generated by the Toeplitz point $x\in X$. By Theorem~\ref{Theorem: S-adic Toeplitz Representation}, there is a sequence of alphabets $\{\mathcal{A}_i\}$ and a sequence of constant length substitutions $\theta_i: \mathcal{A}_i\to\mathcal{A}_{i-1}$ that generate $(X,T)$. For every $k\in \N$ and $1\leq i\leq |\mathcal{A}_k|$, let $B_i^k$ denote the set of points $x\in X$ such that $x_{[0, |\theta_1\circ \theta_2\circ \cdots \circ\theta_k|-1]} = \theta_1\circ\theta_2\circ\cdots\circ\theta_k(i)$ and set $h_i^k = |\theta_1\circ \theta_2\circ \cdots \circ\theta_k| = p_k$. The sequence of KR-partitions is defined by the sets $\mathcal{P}(k)=\bigcup_{i=1}^{|\mathcal{A}_k|}\{T^jB_i^k: 0\leq j\leq p_k - 1 \}$.
\end{definition}

\begin{eg}
    Let $\mathcal{A}= \mathcal{A}_0 = \{0,1\}$, $\mathcal{A}_1 = \{a,b\}$, and $\mathcal{A}_i = \{c,d\}$ for all $i\geq 2$. Define $\theta_1(a) = 011$, $\theta_1(b) = 100$, $\theta_2(c) = ab$ and $\theta_2(d) = aa$, and for all $i\geq 2$, define $\theta_i (c) = ccd$, and $\theta_i(d) = cdd$. Then $\mathcal{P}(3)$ will consist of $2$ towers, each of height $p_3=18$; the first tower will have its base floor given by the set of points $x\in X$ with  $x_{[0,17]} = 011100011100011011$ and the second tower will have its base floor given by the cylinder set $x_{[0,17]}=011100011011011011$.
\end{eg}

We can similarly define the KR-partitions $\mathcal{P}(k)$ when the Toeplitz flow is generated by a substitution system. This is a special case of the Toeplitz definition where $\theta_i = \theta$ for all $i\in \N$.
\begin{definition}
    Suppose that $(X,T)$ is generated by a left proper, primitive, aperiodic substitution on an alphabet $\mathcal{A} = \{1,2,\ldots, N\}$. For every $k\in \N$ and $1\leq i \leq N$, let $B_i^k$ denote the set of points $x\in X$ such that $x_{[0, |\theta|^k-1]} = \theta^k(i)$ and set $h_i^k = |\theta^k(i)|$. Then the sequence of KR-partitions is defined by the sets $\mathcal{P}(k) = \bigcup_{i=1}^N\{T^jB_i^k: 0\leq j \leq h_i^k-1\}$. 
\end{definition}

We note that intuitively, given any minimal Cantor system $(X,T)$ we may think of the KR-partition $\mathcal{P}(k+1)$ as being a collection of towers, where each tower is comprised of a stack of towers from $\mathcal{P}(k)$ and the lowest $\mathcal{P}(k)$-tower in each $\mathcal{P}(k+1)$-tower is the same primary tower from $\mathcal{P}(k)$.

\subsection{Speedups of Minimal Cantor Systems}

\begin{definition}
    Let $(X,T)$ be a minimal Cantor system. A {\em bounded speedup} of $(X,T)$ is a homeomorphism of the form $(X,S)$ where $S(x) = T^{p(x)}(x)$ for some bounded function $p:X\to \N$. We refer to the bounded function $p(x)$ as the {\em jump function}.
\end{definition}

In general, one can explore speedups where $p$ is not assumed to be bounded; however, the function $p$ is known to be continuous if and only if it is bounded \cite[Propostion 2.2]{AAO}. Additionally, the resulting speedup of a minimal Cantor system need not be minimal. If all $S$-orbits are dense in $X$, then we say that $(X,S)$ is a {minimal speedup} of $(X,T)$. Throughout this paper, we will always assume that $p$ is bounded and that $(X,S)$ is a minimal speedup. Given a minimal bounded speedup $(X,S)$ of $(X,T)$, there is a constant $c\in \N$, called the {\em orbit number}, such that every $T$-orbit is the union of exactly $c$ different $S$-orbits \cite[Lemma 2.3]{AAO}.
We now discuss some properties of jump functions as they relate to the sequence of KR-partitions for the underlying minimal Cantor system $(X,T)$.

\begin{remark}\label{Remark: KR levels}
Let $(X,T)$ be a minimal Cantor system. Suppose that $p:X\to \N$ defines a bounded topological speedup $(X,S)$. Then the sequence of KR-partitions for $(X,T)$ is such that for all $k$ sufficiently large:
\begin{enumerate}
    \item  The height $h_i^k$ of each tower $A_i^k$ is greater than $\max_{x\in X} p(x)$.
    \item $p$ is constant on { each floor} of $\mathcal{P}(k)$.
    \item $p$ is constant on each set $T^m(B(k))$ for $0\leq m\leq \max_{x\in X} p(x)$.
\end{enumerate}  By beginning our sequence of KR-partitions with $k$ large enough, we may assume these properties are satisfied for our entire sequence of KR-partitions.
\end{remark}
We may thus create an orbit labeling of the levels in $\mathcal{P}(k)$ that models the speedup. This process demonstrates how every $T$-orbit decomposes into $c$ different $S$-orbits.

\begin{definition}
    For all $k$ large enough so that the properties in Remark~\ref{Remark: KR levels} hold and for $c$ equal to the orbit number of the speedup, let $\mathcal{L}_k: \mathcal{P}(k)\to \{1,2,\ldots,c\}$ be the {\em orbit labeling function} defined as follows. For each $1\leq i \leq N(k)$:
    \begin{itemize}
        \item Label the base floor $B_i^k$ of tower $A_i^k$ with a $1$.
        \item Label any other floor $F$ in $A_i^k$ with a $1$ if $F = S^j(B_i^k)$ and $\sum_{l=0}^{j-1} {  p(S^l(x))} < h_i^k$ for all $x\in B_i^k$.
        \item Label the lowest unlabeled floor $T^{d_2}(B_i^k)$ of $A_i^k$ with a $2$.
        \item Label any other floor $F$ in $A_i^k$ with a $2$ if $F = S^j(T^{d_2}(B_i^k))$ and $\sum_{l=0}^{j-1} { p(S^l(x))} < h_i^k$ for all $x\in T^{d_2}(B_i^k)$.
        \item Label the lowest unlabeled floor in $A_i^k$ with a $3$ and continue.
        \item Continue in this manner until all floors of $A_i^k$ are labeled with a value from $\{1,2,\ldots, c\}$.
    \end{itemize}
\end{definition}

Note that by assuming the properties in Remark~\ref{Remark: KR levels} hold, we guarantee that any two floors in different columns { which are the same number of floors from the bottom and} below the $\max p(x)$ floor will have the same labeling. Suppose that $F$ is the maximal floor in $A_i^k$ that is labeled $l$; then for all $x,y\in F$, the label of the floor containing $S(x)$ is the same as the label of the floor containing $S(y)$.  Hence, for $k$ large enough so that the properties in Remark~\ref{Remark: KR levels} hold, define $\pi_i^{(k)}(l)$ to be the label of the floor containing $S(x)$ for all $x$ in the floor of maximal height in $A_i^k$ labeled $l$.

\begin{proposition}
    If $(X,S)$ is a minimal speedup of $(X,T)$ with jump function $p:X\to \N$, then each $\pi_i^{(k)}$ is a permutation of the set $\{1,2,\ldots, c\}$ {  where $c$ is the orbit number of the speedup $p$ and $c \leq \max_{x\in X} p(x)$}.
\end{proposition}

We now list some known results about speedups of odometers and substitutions that will be helpful in our study of Toeplitz flows. Recall { that for an odometer there exists a KR-partition which consists of a single tower.}

\begin{lemma}\cite[Lemma 3.7]{AAO}\label{Lemma: cyclic permutations}
    Suppose $(X,T) = (X_\alpha,T_\alpha)$ is an odometer and $(X_\alpha,S)$ is a bounded speedup of $(X_\alpha,T_\alpha)$. Then $(X_\alpha, S)$ is minimal if and only if for all $k$ sufficiently large, $\pi^{(k)}$ is a cyclic permutation on $\{1,2,\ldots, c\}$, where $c$ is the orbit number.
\end{lemma}

The following theorem is a combination of several results from \cite{AAO}.

\begin{theorem}\cite[Theorem 3.8]{AAO}\label{Theorem: Odometer Speedup Characterization}
    If $(X,T) = (X_\alpha,T_\alpha)$ is an odometer with $\alpha = (\alpha_1,\alpha_2,\ldots)$, then there is a minimal bounded speedup with orbit number $c>1$ if and only if there is some $N\in \N$ such that $\gcd(c,\alpha_i) = 1$ for all $i\geq N$. Further, in this case the minimal speedup will be topologically conjugate to the original odometer.
\end{theorem}

We note that the Toeplitz flows we are interested in studying are each symbolic shift spaces. It thus follows that any minimal speedup of a Toeplitz flow will necessarily result in a system that is not conjugate to the original system.

\begin{theorem}\cite[Theorem 4.6]{AAO}\label{Theorem: symbolic speedups not conjugate}
    If $(X,T)$ is a minimal subshift on an infinite set $X$ and $(X,S)$ is a minimal bounded speedup of $(X,T)$, then $(X,T)$ and $(X,S)$ are not conjugate.
\end{theorem}

{ 
\begin{lemma}\cite[Lemma 4.14]{AAO}
   For a  minimal bounded speedup of a minimal substitution subshift $(X,T)$ there exists $K\geq 1$ such that $\pi_i^{(K)} = \pi_i^{(jK)}$ for all $1\leq i \leq |\mathcal{A}|$ and all $j\in \N$.
\end{lemma}
}
Let $\theta$ be a constant length, left proper, primitive substitution, { because of the previous lemma,} we may assume that $K$ is chosen large enough to satisfy Remark~\ref{Remark: KR levels}. Additionally, because the system generated by $\theta$ is conjugate to the system generated by $\theta^k$, without loss of generality we assume that $\theta$ satisfies all of these hypotheses with $k=1$ and $K=1$.

\begin{definition}\label{Definition: tau}
    Let $\theta$ be a constant length, left proper, primitive substitution on the alphabet $\mathcal{A}$. 
    Set $C = \{1,2,\ldots, c\}$ and consider the alphabet $\mathcal{A}\times C$. Define $\tau: \mathcal{A}\times C\to (\mathcal{A}\times C)^*$ by $$\tau(a,\ell) = (a_1,\ell_1)(a_2,\ell_2)\cdots(a_n,\ell_n)$$ where $\theta(a) = a_1a_2\cdots a_n$, $\ell_1 = \ell$, and $\ell_{k+1} = { \pi_{a_k}}(\ell_k)$ for each $k\geq 1$.
\end{definition}

\begin{lemma}\cite[Lemma 4.15]{AAO}\label{Lemma: tau is primitive}
    Suppose $\theta$ is a (left) proper, primitive, aperiodic substitution on $\mathcal{A}$ and $(X,T)$ is the subshift generated by $\theta$. Suppose $(X,S)$ is a bounded speedup of $(X,T)$ with orbit number $c$. Then $(X,S)$ is minimal if and only if $\tau$ is primitive.
\end{lemma}

\begin{theorem}\cite[Theorem 4.18]{AAO}
    Suppose $(X,T)$ is a minimal substitution system associated with a (left) proper primitive substitution $\theta$. If $(X,S)$ is a minimal bounded speedup of $(X,T)$, then $(X,S)$ is a minimal substitution system.
\end{theorem}

We now generalize Definition~\ref{Definition: tau} and Lemma~\ref{Lemma: tau is primitive} from the substitution setting to the more general S-adic setting for Toeplitz flows.
\begin{definition}\label{Definition: generalize tau}
    Let $(\mathcal{A}_i)_{i\geq 0}$ be a sequence of alphabets and $(\theta_i)_{i\geq 1}$ be a sequence of constant length, left proper, primitive substitutions such that $\theta_i: \mathcal{A}_i \to\mathcal{A}_{i-1}^*$. Set $C = \{1,2,\ldots, c\}$ and consider the sequence of alphabets $((\mathcal{A}_i\times C)_{i\geq 0})$. { For $k>l$ define} $\tau_k^l: \mathcal{A}_k\times C \to (\mathcal{A}_l\times C)^*$ by
    $$
    \tau_k^l(a,\ell) = (a_1,\ell_1)(a_2,\ell_2)\cdots (a_n,\ell_n)
    $$
    where $\theta_{l+1}\circ\theta_{l+2}\circ \cdots \circ \theta_k(a) = a_1a_2\cdots a_n$, $\ell_1 = \ell$, and $\ell_{j+1} = { \pi_{a_j}}(\ell_j)$ for each $j\geq 1$.
\end{definition}
\begin{lemma}\label{Lemma: generalize tau is primitive}
    Let $(X,T)$ be the Toeplitz flow generated by the sequences $(\mathcal{A}_i)$ and $(\theta_i)$. Suppose $(X,S)$ is a bounded speedup of $(X,T)$ with orbit number $c$. Then $(X,S)$ is minimal if and only if for every $l\in \N$ there is a $k\geq l$ such that $\tau_k^l(a,j)$ contains every letter from $\mathcal{A}_l\times C$ where $a\in \mathcal{A}_k$ and $j\in C$ are arbitrarily chosen.
\end{lemma}
\begin{proof}
    Suppose for all $l\in \N$ there exists a $k\geq l$ such that $\tau_k^l(a,j)$ contains every letter from $\mathcal{A}_l\times C$, where $a\in \mathcal{A}_k$ and $j\in C$ are arbitrarily chosen. Let $x\in X$ and consider the decomposition of $x$ into $\theta_1\circ\theta_2\circ \cdots \theta_k$-words.  Within each $\theta_1\circ\theta_2\circ\cdots\circ \theta_k$-word is every possible $\theta_1\circ\theta_2\circ\cdots \theta_l$-word with every possible label.  Thus, the $S$-orbit of $x$ intersects all floors of $\mathcal{P}(l)$. Because this is true for all $l\in \N$, the $S$-orbit of $x$ is dense.

    Conversely, suppose that $S$ is minimal.
    There is an $N$ such that any $S$-orbit block of length $N+1$ intersects every floor of $\mathcal{P}(l)$ with all labels. Set $M = \sup_{x\in X} p(x)$ and let $k>l$ be chosen such that $|\tau_k^l(a,j)| \geq MN$ for all $(a,j)\in \mathcal{A}_k\times C$. Then we may conclude that every letter $(a',j')\in \mathcal{A}_l\times C$ appears in $\tau_k^l(a,j)$ for each $(a,j)\in \mathcal{A}_k\times C$.
\end{proof}

\section{The Speedup of a Toeplitz Flow Need Not be Toeplitz}\label{sec: Not Toeplitz}

In this section we present an example of a Toeplitz flow $(X,T)$ and a jump function $p(x): X\to \N$ such that the speedup $(X,S) = (X,T^{p(x)})$ is minimal and not conjugate to a Toeplitz flow. This demonstrates that the conjecture in \cite{ADL} that the minimal speedup of a Toeplitz flow is Toeplitz is not true.

\begin{eg}\label{Example: New Non Example}
    Consider the primitive substitution $\theta$ on the alphabet $\mathcal{A} = \{a,b\}$ given by $\theta(a) = ab$ and $\theta(b) = aa$. Note that the subshift $(X_\theta, \sigma)$ is a Toeplitz flow. Let $A = [\theta^3(a)]= \{x\mid x_{[0,7]} = abaaabab\}$ and define $p:X_\theta \to \Z^+$ by $$ p(x) = \begin{cases}
    3 & \text{if } x\in A \\
    1 & \text{if } x\in \sigma(A)\\
    2 & \text{else.}
    \end{cases}$$ With this given jump function $p$, $(X_\theta,S)$ is a minimal bounded speedup of $(X_\theta,\sigma)$ where $S = \sigma^{p(x)}(x)$. In this case, the orbit number is $c=2$.
\end{eg}

\begin{remark}\label{Remark: Classical example substitutions}
    For $(X_\theta,\sigma)$, the tower $A_1^k$ is given by $\theta^k(a)$ and the tower $A_2^k$ is given by $\theta^k(b)$. For all $k\geq 3$, the permutation on the first tower is given by $\pi_1^{(k)}(1) = 2$ and $\pi_1^{(k)}(2)=1$, whereas the permutation on the second tower is given by $\pi_2^{(k)}(1)=1$ and $\pi_2^{(k)}(2) = 2$. Because the permutation depends only on the tower and not on the level of the KR-partition, we may simply use $\pi_1(i)$ and $\pi_2(i)$ for all KR partitions $\mathcal{P}(k)$.

    Also observe that the height of orbit $j$ in tower $A_i^k$ is given by $o_i^k(j) = 2^{k-1}$ for all $i,j\in \{1,2\}$ and for all $k\geq 3$. In particular, the height of each orbit in each tower of $\mathcal{P}(3)$ is $4$. To see this more clearly, let $A = aba$, $B = aa$, $C = ba$, $D=b$, and $E=ab$. Then tracing through the first orbit of the tower $A_1^3$ is equivalent to tracing through $ABCC$. Tracing through the second orbit of the tower $A_1^3$ is equivalent to tracing through $DBEE$. Similarly, tracing through the first orbit of the tower $A_2^3$ is equivalent to tracing through the $EBEB$ and tracing through the second orbit of $A_2^3$ is equivalent to tracing through $CBCB$. 

    In order to understand the speedup $(X_\theta, S)$, we first define $\mathcal{C} = \{A,B,C,D,E\}$ and $\mathcal{B} = \{w_1,w_2,w_3,w_4\}$.  We then define two substitutions as follows:
    \begin{itemize}
    \item Define $\phi: \mathcal{B}\to \mathcal{B}^*$ by $\phi(w_1) = w_1w_2w_3w_4$, $\phi(w_2) = w_1w_2w_4w_3$, $\phi(w_3) = w_1w_2w_3w_4w_3$, and $\phi(w_4) = w_1w_2w_4$. 
    \item Define $\psi:\mathcal{B}\to \mathcal{C}^*$ by $\psi(w_1) = ABCCCBCBDBEE$, $\psi(w_2) = ABCCDBEEEBEB$, $\psi(w_3) = ABCCCBCBDBEEEBEB$, and $\psi(w_4) = ABCCDBEE$.
    \end{itemize}
    By using the return word method of Durand \cite{Durand}, it can be shown that the speedup $(X_\theta,S)$ is topologically conjugate to the shift space defined by applying $\psi$ to the fixed point of the substitution $\phi$ and then taking the closure of the shift on this new point. We denote this space $(X_\psi, \sigma)$.
    \end{remark}
    
    { In \cite{Durand}, Durand uses return words to recode substitution subshifts $(X_\theta, \sigma)$. A word $w$ in the language of $X_\theta$ is a {\em return word} to the word $u$ if $wu$ is in the language of $X$, $u$ is a prefix and suffix of $wu$, and there are no other occurrences of $u$ in $w$. 
    Taking the fixed point $x$ of the substitution subshift $(X_\theta, \sigma)$, we rewrite it into the jump words like $A,B,C,D,E$ in Remark~\ref{Remark: Classical example substitutions} and look for return words to the letter $A$. The return words obtained for $(X_\theta,S)$  are $w_1,w_2,w_3$ and $w_4$, and by recoding $\theta(\psi(w_i))$ into the return words we find the substitution $\phi$.}

    \begin{theorem}\label{Theorem: Example not Toeplitz}
        The speedup $(X_\theta,S)$ defined in Example~\ref{Example: New Non Example} is not a Toeplitz flow.
    \end{theorem}
    \begin{proof}
        Suppose that $(X_\theta,S)$ is a Toeplitz flow. Then without loss of generality, there exists some word $y\in X_\theta$ that is Toeplitz under $S$. Without loss of generality, we may assume that the decomposition of $y$ into $\mathcal{C}$-words begins with $y_{[0,8]}=ABCC$. It follows that the decomposition of $y$ into these $\mathcal{C}$-words is a Toeplitz sequence. 
        Hence there exists some $t\in \N$ such that $A$ appears in the $\mathcal{C}$-decomposition of $y$ in every $t$ positions. Further, because $y$ is necessarily minimal and the substitution $\phi$ is recognizable, $\psi(\phi^k(w_2))$ appears in this decomposition of $y$ for every $k\in \N$. Hence, $A$ must appear in $\psi(\phi^k(w_2))$ exactly every $t$ positions for every $k\in \N$.

\newcommand\scalemath[2]{\scalebox{#1}{\mbox{\ensuremath{\displaystyle #2}}}}

        Observe the following $\psi(\phi)$ words (formatted with spaces to ease identifying the locations of $A$). \[\psi(\phi(w_1)) = 
        \scalemath{.9}{ABCCCBCBDBEE~ABCCDBEEEBEB~ ABCCCBCBDBEEEBEB~ABCCDBEE}\]
        $$\psi(\phi(w_2)) = \scalemath{.9}{ABCCCBCBDBEE~ABCCDBEEEBEB~ABCCDBEE~ABCCCBCBDBEEEBEB}$$
        \begin{align*}
        \psi(\phi(w_3)) &= \scalemath{.9}{ABCCCBCBDBEE~ABCCDBEEEBEB~ABCCCBCBDBEEEBEB~ABCCDBEE} \\
        &\quad \ \scalemath{.9}{ABCCCBCBDBEEEBEB}
        \end{align*}
        $$\psi(\phi(w_4) = \scalemath{.9}{ABCCCBCBDBEE~ABCCDBEEEBEB~ABCCDBEE \qquad\qquad\qquad\qquad\qquad\qquad \quad }$$
        By inspection, it necessarily follows that $t\geq 24$.  Additionally, we can look at the $\psi(\phi^2)$ words.
        \begin{align*}
            \psi(\phi^2(w_1))& = \psi(w_1w_2w_3w_4w_1w_2w_4w_3w_1w_2w_3w_4w_3w_1w_2w_4) \\
            & = \scalemath{.9}{ABCCCBCBDBEE~ABCCDBEEEBEB~ABCCCBCBDBEEEBEB~ABCCDBEE} \\
            &\quad \ \scalemath{.9}{ABCCCBCBDBEE~ABCCDBEEEBEB~ABCCDBEE~ABCCCBCBDBEEEBEB}\\
             &\quad \ \scalemath{.9}{ABCCCBCBDBEE~ABCCDBEEEBEB~ABCCCBCBDBEEEBEB~ABCCDBEE} \\
             &\quad \ \scalemath{.9}{ABCCCBCBDBEEEBEB ~ ABCCCBCBDBEE~ABCCDBEEEBEB~ABCCDBEE}
        \end{align*}
         \begin{align*}
            \psi(\phi^2(w_2))& = \psi(w_1w_2w_3w_4w_1w_2w_4w_3w_1w_2w_4w_1w_2w_3w_4w_3) \\
            & = \scalemath{.86}{ABCCCBCBDBEE~ABCCDBEEEBEB~ABCCCBCBDBEEEBEB~ABCCDBEE} \\
            &\quad \ \scalemath{.86}{ABCCCBCBDBEE~ABCCDBEEEBEB~ABCCDBEE~ABCCCBCBDBEEEBEB}\\
             &\quad \
             \scalemath{.86}{ABCCCBCBDBEE~ABCCDBEEEBEB~ABCCDBEE~
             ABCCCBCBDBEE} \\
             &\quad \ \scalemath{.86}{ABCCDBEEEBEB~ABCCCBCBDBEEEBEB~ABCCDBEE~ ABCCCBCBDBEEEBEB} 
        \end{align*}
        \begin{align*} 
            \psi(\phi^2(w_3))& = \psi(w_1w_2w_3w_4w_1w_2w_4w_3w_1w_2w_3w_4w_3w_1w_2w_4w_1w_2w_3w_4w_3) \\
             &  = \scalemath{.9}{ABCCCBCBDBEE~ABCCDBEEEBEB~ABCCCBCBDBEEEBEB~ABCCDBEE} \\
            &\quad \ \scalemath{.9}{ABCCCBCBDBEE~ABCCDBEEEBEB~ABCCDBEE~ABCCCBCBDBEEEBEB}\\
             &\quad \ \scalemath{.9}{ABCCCBCBDBEE~ABCCDBEEEBEB~ABCCCBCBDBEEEBEB~ABCCDBEE} \\
             &\quad \ \scalemath{.9}{ABCCCBCBDBEEEBEB ~ ABCCCBCBDBEE~ABCCDBEEEBEB~ABCCDBEE} \\
             &\quad \ \scalemath{.9}{ABCCCBCBDBEE~ABCCDBEEEBEB~ABCCCBCBDBEEEBEB~ABCCDBEE} \\
             &\quad \ \scalemath{.9}{ABCCCBCBDBEEEBEB},
        \end{align*}
        \begin{align*}
            \psi(\phi^2(w_4))& = \psi(w_1w_2w_3w_4w_1w_2w_4w_3w_1w_2w_4) \\
            & = \scalemath{.9}{ABCCCBCBDBEE~ABCCDBEEEBEB~ABCCCBCBDBEEEBEB~ABCCDBEE} \\
            &\quad \ \scalemath{.9}{ABCCCBCBDBEE~ABCCDBEEEBEB~ABCCDBEE~ABCCCBCBDBEEEBEB}\\
             &\quad \
             \scalemath{.9}{ABCCCBCBDBEE~ABCCDBEEEBEB~ABCCDBEE},
        \end{align*}
        
        By applying the same logic as above, we see that necessarily $t\geq 96$.
        Suppose that for some $k\geq 1$, $A$ appears every $t$ positions in $\psi(\phi^k(w_2))$ for $t = 6\cdot 4^k$. We now outline the following facts about $\psi\circ\phi^k$ and $\psi\circ\phi^{k+1}$.
        \begin{enumerate}
        \item Facts about the lengths of the $\psi(\phi^k(w))$ words:
        \begin{itemize}
            \item $|\psi(\phi^k(w_1))| = 6\cdot 4^k + 6\cdot 4^k$.
            \item $|\psi(\phi^k(w_2))| = 6\cdot 4^k + 6\cdot 4^k$
            \item $|\psi(\phi^k(w_3))| = 6\cdot 4^k + 6\cdot 4^k + 4\cdot4^k$
            \item $|\psi(\phi^k(w_4))| = 6\cdot 4^k + 2\cdot 4^k$
        \end{itemize}
        \item Facts about concatenation of words in $\psi(\phi^{k+1}(w_2))$:
        \begin{itemize}
            \item $\psi(\phi^{k+1}(w_2)) = \psi(\phi^k(w_1)\phi^k(w_2)\phi^k(w_4)\phi^k(w_3))$
            \item $|\psi(\phi^{k+1}(w_2))| = |\psi(\phi^k(w_1)\phi^k(w_2)\phi^k(w_4)\phi^k(w_3))| = 6\cdot 4^k + 6\cdot 4^k + 6\cdot 4^k + 6\cdot 4^k + 6\cdot 4^k + 2\cdot 4^k + 6\cdot 4^k + 6\cdot 4^k + 4\cdot 4^k $
            \item The alignment of $\psi(\phi^k(w_4)\phi^k(w_3))$ is such that $A$ does not appear at all multiples of $6\cdot 4^k$. It does hold that $A$ appears at all multiples of $6\cdot 4^{k+1}$ in $\psi(\phi^k(w_4)\phi^k(w_3))$. Observe that $A$ will also appear at all multiples of $6\cdot 4^{k+1}$ in $\psi(\phi^k(w_1w_2))$.
            \item We thus conclude that $A$ appears at all multiples of $6\cdot 4^{k+1}$ in $\psi(\phi^{k+1}(w_2))$. Therefore, it follows that if $A$ appears every $t$ positions in the $\mathcal{C}$-decomposition of $y$, then $t\geq 6\cdot 4^{k+1}$.
        \end{itemize}
        \end{enumerate}
        
        For each $k\in \N$, it follows that in order for $A$ to appear within $\psi(\phi^k(w_2))$ at every $t$ positions, { $t$ must be at least $6\cdot 4^k$}. Since $t$ is assumed to be a finite number that works for every $k\in \N$, we obtain a contradiction. It thus follows that the decomposition of $y$ into $\mathcal{C}$-words is not a Toeplitz sequence, and hence $y$ will not be a Toeplitz sequence under $S$. We conclude that no point in $X_\theta$ will be Toeplitz under $S$, and therefore $(X_\theta, S)$ is not a Toeplitz flow.
    \end{proof}

We now show that no minimal speedup of the Toeplitz flow $(X_\theta, \sigma)$, where $\theta$ is defined as in Example~\ref{Example: New Non Example}, with orbit number $c=2$ is a Toeplitz flow. 
    \begin{theorem}\label{Theorem: no speedup with c=2 is Toeplitz}
        Any minimal speedup of the Toeplitz flow $(X_\theta,\sigma)$ generated by the substitution $\theta(a) = ab$ and $\theta(b) = aa$ with orbit number $c=2$ { is not} a Toeplitz flow.
    \end{theorem}
    \begin{proof}
        Let $p(x)$ be a jump function such that $(X_\theta, S)$ is a minimal bounded speedup of $(X_\theta,\sigma)$ with orbit number $c=2$. Choose a KR-partition  $\mathcal{P}(k)$ with $k$ chosen large enough such that the properties in Remark~\ref{Remark: KR levels} hold. Consider the permutations $\pi_a^{(k)}$ and $\pi_b^{(k)}$ induced by the speedup. Observe that if $\pi_a^{(k)}=\pi_b^{(k)} = \text{id}$, then the associated speedup would not be minimal.  Similarly, if $\pi_a^{(k)}(1) = \pi_b^{(k)}(1) = 2 $, the associated speedup will not be minimal as $\pi_a^{(k+1)}=\pi_b^{(k+1)} = \text{id}$. Thus, one of { the} tower permutations must be the identity, and  { the} other tower permutation must be the cyclical permutation. Observe that in either case, for all $k\geq K$ (there exists $K\in \N$), it follows that $\pi_a^{(k)}(1) = 2$ and $\pi_b^{(k)} = \text{id}$. One can check via Lemma~\ref{Lemma: tau is primitive} that the associated speedup will be minimal. 

        Let $o_a^k(i)$ denote the number of levels of the $\theta^k(a)$ tower in the KR-partition $\mathcal{P}(k)$ that are labeled with $i$ for each $i\in \{1,2\}$; similarly define $o_b^k(i)$. Observe that $o_a^k(1)+o_a^k(2) = |\theta^k|$ and $o_b^k(1) + o_b^k(2) = |\theta^k|$. Further note that the substitution $\phi: \mathcal{B}\to \mathcal{B}^\ast$ defined as in Remark~\ref{Remark: Classical example substitutions} is the same as the substitution defined by checking returns to $(a,1)$ in the substitution defined by $\tau$ in Definition~\ref{Definition: tau}. Here $\mathcal{B} = \{w_1 = (a,1)(b,2)(a,2), w_2=(a,1)(a,2)(b,1), w_3 = (a,1)(b,2)(a,2)(b,1), w_4=(a,1)(a,2)\}$, $\phi(w_1) = w_1w_2w_3w_4$, $\phi(w_2) = w_1w_2w_4w_3$, $\phi(w_3) = w_1w_2w_3w_4w_3$, and $\phi(w_4) = w_1w_2w_4$. Observe that when we consider the orbit heights of each orbit path, the length of $w_1$ is $o_a^k(1)+o_b^k(2)+o_a^k(2) = |\theta|^k + o_b^k(2)$, the length of $w_2$ is $o_a^k(1)+o_a^k(2)+o_b^k(1) = |\theta|^k + o_b^k(1)$, the length of $w_3$ is $o_a^k(1) + o_b^k(2) + o_a^k(2) + o_b^k(1) = 2|\theta|^k$, and the length of $w_4$ is $o_a^k(1) + o_a^k(2) = |\theta|^k$.

        We now follow the same logic as in the previous theorem to demonstrate that $(X_\theta, S)$ is not Toeplitz.
        Suppose that $(X_\theta, S)$ is Toeplitz. Then there is some sequence $x\in X_\theta$ that is Toeplitz under $S$. We can decompose $x$ uniquely into its $w_i$ representation. As $x$ is Toeplitz, there exists some $l\in \N$ such that if the decomposition of $x$ into $w_i$ words begins with $w_1$, then the decomposition of $S^{l\cdot n}(x)$ begins with $w_1$ for all $n\in \N$. By our substitution, $|\phi(w_1)| = |\phi(w_2)| = 6|\theta|^k$, $|\phi(w_3)| = 8|\theta|^k$,and $|\phi(w_4)| = 4|\theta|^k$.  Hence, any return between our desired occurrences of $w_1$ in the $w_i$-decomposition of $x$ gives $l \geq 8|\theta|^k$.  By looking at $\phi^2(w_1) = w_1w_2w_3w_4w_1w_2w_4w_3w_1w_2w_3w_4w_3w_1w_2w_4$, we obtain that $l\geq 12|\theta|^k$.  By looking at $\phi^3(w_1)$, we obtain that $l\geq 24|\theta|^k$. In general, we will never be able to find a single value $l$ such that our $w_i$-decomposition of $x$ returns to $w_1$ every $l$ steps. Therefore $(X_\theta, S)$ is not a Toeplitz flow.
    \end{proof}

    {
    The proofs of the prior two theorems are meant to provide a strategy that can be used to determine whether a given speedup is Toeplitz; this strategy will be formalized in the following section. We note that while the speedup $(X_\theta,S)$ in Example~\ref{Example: New Non Example} is not a Toeplitz flow, it is conjugate to the Thue-Morse substitution shift.
    
    \begin{theorem}
        The speedup $(X_\theta,S)$ defined in Example~\ref{Example: New Non Example} is conjugate to the shift space generated by the Thue-Morse substitution.
    \end{theorem}
    \begin{proof}
        Define $\mathcal{B} = \{w_1,w_2,w_3,w_4\}$, $\mathcal{C} = \{A,B,C,D,E\}$, and $\mathcal{D} = \{0,1\}$. Further, let $\phi: \mathcal{B}\to \mathcal{B}^\ast$ and $\psi: \mathcal{B}\to \mathcal{C}^\ast$ be as in Remark~\ref{Remark: Classical example substitutions} and define $\gamma: \mathcal{D}\to \mathcal{D}^\ast$ by $\gamma(0) = 01$ and $\gamma(1) = 10$ to be the Thue-Morse substitution. We 
        observe that one of the one-sided fixed points of the substitution $\gamma$ is $$\bfx= \lim_{n\to\infty}\gamma^n(0) = 0110 1001 1001 0110 1001 0110 0110 1001 \cdots.$$
        By defining $\bfu = \lim_{n\to \infty} \phi^n(w_1) = w_1w_2w_3w_4w_1w_2w_4w_3w_1w_2w_3w_4w_3 \cdots$ and applying $\psi$ letter-by-letter, we obtain
            $$\bfy = \psi(\bfu) = \psi(w_1)\psi(w_2)\psi(w_3)\cdots = ABCC CBCB DBEE ABCC DBEE EBEB \cdots.$$
        The associated shift spaces generated by these substitutions are the orbit closures of the sequences $\bfx$ and $\bfy$ under the shift map $\sigma$. 
        We denote by $X_\gamma$ the shift space generated by the Thue-Morse substitution, whereas $X_\psi$ is precisely the shift space generated by the speedup $(X_\theta,S)$.

        We define a function $F: \mathcal{C}^3\to \mathcal{D}$ that sends the allowed triples appearing in $X_\psi$ to either $0$ or $1$ according to the following chart.

        \begin{center}
        \begin{tabular}{cc|cc|cc|cc}
            \hline
            Triple & $F$ & Triple & $F$ & Triple & $F$ & Triple & $F$ \\
            \hline
            $ABC$ & 0 & $CCB$ & 0 & $BCC$ & 1 & $CCD$ & 1 \\
            $BAB$ & 0 & $CDB$ & 0 & $BDB$ & 1 & $DBE$ & 1 \\
            $BCB$ & 0 & $EBE$ & 0 & $BEB$ & 1 & $EAB$ & 1 \\
            $BEE$ & 0 & $EEA$ & 0 & $CBC$ & 1 & $EBA$ & 1 \\
            $CBD$ & 0 & $EEE$ & 0 & $CCC$ & 1 & $EEB$ & 1 \\
            \hline
        \end{tabular}
        \end{center}
        Then the desired conjugacy between $X_\psi$ and $X_\gamma$ is given by the sliding block code $\mathcal{F}: X_\psi \to X_\gamma$ defined by $\mathcal{F}(\bfy)_n = F(y_ny_{n+1}y_{n+2})$. We conclude that the speedup $(X_\theta,S)$ is conjugate to the Thue-Morse substitution shift space.
    \end{proof}

    As the Thue-Morse substitution shift is an almost 2-1 extension of the dyadic odometer, we have that the speedup $(X_\theta,S)$ from Example~\ref{Example: New Non Example} is an almost 2-1 extension of an odometer.
    }
    This leads to the following natural question.
    \begin{question}\label{Question: constant length substitutions no orbit number constant length}
        Suppose $(X_\theta, \sigma)$ is an arbitrary Toeplitz flow generated by the constant length substitution $\theta$. Is there any minimal speedup $(X_\theta, S)$ of $(X_\theta, \sigma)$ with orbit number $c = |\theta|$ such that $(X_\theta, S)$ is Toeplitz? { If not, will the resulting speedup be an almost $c$-1 extension of an odometer?} What if, more generally, the orbit number $c$ is such that $\gcd(c, |\theta|) \neq 1$?
    \end{question}

\section{Proving that a Speedup is (Not) Toeplitz}\label{Sec: Determining if Toeplitz}

{ In this section we characterize when a constant speedup $(X,T^c)$ of a Toeplitz flow $(X,T)$ is minimal, Toeplitz, and over the same underlying odometer. The key tool is the almost 1-1 extension structure of Toeplitz flows over odometers. We conclude the section with a strategy to determine when a non-constant minimal bounded speedup of a Toeplitz flow is Toeplitz.

\begin{theorem}\label{Theorem: constant speedup minimality characterization}
   Let $(X,T)$ be a Toeplitz flow with underlying odometer $(G,\tau)$ and period structure $(p_k)_{k\geq 1}$.
   \begin{enumerate}
       \item $(X,T^c)$ is minimal if and only if $\gcd(c,p_k)=1$ for all $k$.
       \item If $(X,T^c)$ is minimal, then it is a Toeplitz flow with the same underlying odometer as $(X,T)$.
       \item If $\gcd(c,p_k)>1$ for some $k$, then $(X,T^c)$ is not minimal. Each minimal component of $(X,T^c)$ is a mutually conjugate Toeplitz flow.
   \end{enumerate}
\end{theorem}
\begin{proof}
Let $h:X\to G$ denote the factor map from $(X,T)$ onto its maximal equicontinuous factor $(G,\tau)$; then $h$ is almost 1-1. As $h \circ T = 
\tau\circ h$, we have $h\circ T^c = \tau^c\circ h$, and thus $h$ is an almost 1-1 map from $(X,T^c)$ to $(G,\tau^c)$. Observe that $(G,\tau^c)$ is minimal if and only if $\gcd(c,p_k) = 1$ for all $k$, in which case $(G,\tau^c)$ is conjugate to $(G,\tau)$, and hence $(X,T^c)$ is a Toeplitz flow over the same odometer. If $\gcd(c,p_k)\neq 1$ for some $k$, then $(G,\tau^c)$ splits into finitely many mutually conjugate odometers, and thus $(X,T^c)$ splits into finitely many mutually conjugate Toeplitz flows.
\end{proof}


}

{ The following example demonstrates that while the underlying odometer is preserved by a constant speedup, the period structure may change. Specifically, $(p_k)$ need not be a period structure for $(X,T^c)$, even though both systems have the same underlying odometer. }

\begin{eg}\label{Example: period structure changes under speedup}
    Consider the primitive substitution $\theta$ on the alphabet $\mathcal{A} = \{a,b\}$ given by $\theta(a) = ab$ and $\theta(b) = aa$. The subshift $(X_\theta, T)$ is a Toeplitz flow with period structure $(p_k) = (2^k)$ for all $k\geq 1$. Consider the constant speedup $(X_\theta, T^c)$ where $c=3$. By Theorem~\ref{Theorem: constant speedup minimality characterization}, $(X_\theta, T^3)$ is a Toeplitz flow. One can verify that $(X_\theta, T^3)$ is conjugate to the substitution shift generated by $\gamma: \{A,B,C,D,E\}\to \{A,B,C,D,E\}^\ast$ given by $\gamma(A) = ABAC$, $\gamma(B) = ABDE$, $\gamma(C) = AEDC$, $\gamma(D) = ABDC$, and $\gamma(E) = AEDE$.  Observe that this substitution is once again constant length and left proper. One associated period structure for $(X_\theta, T^3)$ is $(q_k) = (4^k)$ for every $k\in \N$. In this case, we note that the underlying odometer for both $(X_\theta, T)$ and $(X_\theta,T^3)$ is the dyadic odometer; in fact $(q_k)$ is a period structure for $(X_\theta,T)$, even though $(p_k)$ is not a period structure for $(X_\theta, T^3)$.
\end{eg}

{ This example also demonstrates that a minimal constant speedup of a substitution subshift generated by a constant length and left proper substitution can be conjugate to a substitution subshift generated by another constant length and left proper substitution. We now show that this is true in general. 

\begin{theorem}\label{Theorem: constant speedup substitution}
    Let $\theta$ be a primitive, left proper, and constant $q$-length substitution on the alphabet $\mathcal{A}$ and $(X_\theta,T)$ be the associated substitution shift. If $\gcd(c,q)=1$, then $(X_\theta,T^c)$ is conjugate to a substitution shift $(X_\gamma,\sigma)$ generated by a proper, primitive, and constant length substitution $\gamma$.
\end{theorem}


\begin{proof}
Let $\mathcal{B}$ denote the set of all words of length $c$ appearing in the language of $(X_\theta,T)$, which will serve as the alphabet of the new substitution. Since $\theta$ is left proper and primitive, we may choose $k\in\mathbb{N}$ sufficiently large that $q^k\ge c$ and the word $\theta^k(a)$ begins with the same common prefix $w^*\in\mathcal{B}$ of length $c$ for every $a\in\mathcal{A}$.

For each letter $B=b_1b_2\cdots b_c\in\mathcal{B}$, consider the image $\theta^k(B)=\theta^k(b_1)\theta^k(b_2)\cdots\theta^k(b_c)$, which has length $c\cdot q^k$. Partition this word into $q^k$ consecutive non-overlapping blocks of length $c$, writing $\theta^k(B)=W_1W_2\cdots W_{q^k}$ with $|W_j|=c$ for each $j$. Each block $W_j$ is a subword of length $c$ in the language of $(X_\theta,T)$, so $W_j\in\mathcal{B}$. We define the substitution $\gamma:\mathcal{B}\to\mathcal{B}^*$ by $\gamma(B)=W_1W_2\cdots W_{q^k}$.
This substitution $\gamma$ is constructed to be proper, primitive, and constant length, as it inherits these properties from $\theta$ and the fact that $\gcd(c,q)=1$. As the above procedure was obtained as an adaptation of Durand's return word method, it is easy to verify that $(X_\theta,T^c)$ is conjugate to $(X_\gamma, \sigma)$.
\end{proof}


In general, a minimal speedup of a constant length substitution need not be conjugate to another constant length substitution. However, every example we have constructed (see also Section~\ref{Sec: Sufficient Conditions}) where the minimal speedup of a Toeplitz flow generated by a constant length left proper substitution is again a Toeplitz flow has the property that the resulting speedup can be represented as a constant length and left proper substitution, even when the jump function is not constant. This leads us to ask if this is always the case.}



\begin{question}
    If $(X,T)$ is a Toeplitz flow generated from a constant length and left proper substitution, and the speedup $(X,S)$ is also a Toeplitz flow, must there be a constant length representation of the substitution generating $(X,S)$?
\end{question}

 We now generalize the technique used in the proofs of Theorem~\ref{Theorem: Example not Toeplitz} and Thereom~\ref{Theorem: no speedup with c=2 is Toeplitz} into a strategy to determine whether a given minimal bounded speedup of a Toeplitz flow is Toeplitz. We present the technique in two versions: the constant length substitutive Toeplitz case and the general (non-substitutive) Toeplitz case.

\medskip

\newpage
{\bf { Strategy} to determine whether a speedup of a substitutive Toeplitz flow is Toeplitz:}
\begin{enumerate}
    \item Given the constant length and left proper substitution $\theta$ and the jump function $p(x)$, find the tower permutations $\pi_i^{(k)}$ for the speedup. Observe that we may choose $k$ large enough and consider the substitution $\theta^k$ so that without loss of generality, every KR-partition $\mathcal{P}(k)$ is such that the $i$th tower at every level has the same permutation and Remark~\ref{Remark: KR levels} is satisfied.
    \item Define the substitution $\tau: \mathcal{A}\times C \to (\mathcal{A}\times C)^\ast$ as in Definition~\ref{Definition: tau}. Note that this substitution can be rewritten as an equivalent substitution $\phi: \mathcal{B}\to \mathcal{B}^*$ by taking the fixed point of the substitution $\tau$ generated by $(a,1)$ and using Durand's method (see \cite{Durand}) to create a left proper substitution by looking at returns to $(a,1)$. 
    \item Define a substitution $\psi$ that encodes the individual jumps within each tower orbit into a new alphabet $\mathcal{C}$; this substitution extends to the $\phi$-words.
    \item Take a bi-infinite fixed sequence $x$ under $\phi$ and apply $\psi$. Then the shift space generated by $\psi(x)$ is conjugate to the speedup $(X,S)$. Note that $x$ can be found by extending the labeling rule defined in Definition~\ref{Definition: tau} to the bi-infinite fixed point of $\theta$ where we assume $\ell_0 = 1$, $\ell_k = \pi_{i_k}(\ell_{k-1})$ if $k\geq 1$ and $\ell_j = \pi_{i_j}^{-1}(\ell_{j+1})$ if $j\leq -1$. 
    \item Observe if $\psi\circ \phi$ has a constant length and left proper representation, then immediately $(X,S)$ is Toeplitz.
    \item Suppose that some point $y\in X$ is Toeplitz under $S$. There is a unique decomposition of $y$ into $\psi(\phi^k)$-words for every $k\in\N$. We call this decomposition $y'$. Since the $0$th coordinate of $y'$ appears every $p$ positions (there exists some $p\in\N$, since $y'$ is Toeplitz under the shift map), it follows that this symbol must appear every $p$ positions in every $\psi(\phi^k)$ word for every $k\in\N$.
    \item Thus, if one can show that no such $p$ can possibly exist (through checking larger and larger $k$ values), we obtain a contradiction and $(X,S)$ is not Toeplitz.
    \item If it follows that there is a $p\in \N$ such that the $0$th coordinate of $y'$ appears every $p$ positions in $y'$, then we set $p = p_1$ and we suppose there is some $p_2\in \N$ (necessarily a multiple of $p_1$) such that the word $y'_{[0,p_1-1]}$ appears every $p_2$ positions in every $\psi(\phi^k)$ word for every $k\in \N$ and continue recursively. If we can find some infinite sequence $(p_k)$ of integers, this will form our period structure for $y'$ under the shift map and we will have verified that $y$ is Toeplitz under $S$. If at any stage we cannot find a valid $p_k$, then the speedup will not be Toeplitz.
\end{enumerate}
\medskip

\newpage
{\bf { Strategy} to determine whether a speedup of a general Toeplitz flow is Toeplitz:}
    \begin{enumerate}
        \item Given the sequence of alphabets $(\mathcal{A}_i)$ and the sequence of constant length left proper substitutions $(\theta_i: \mathcal{A}_i\to \mathcal{A}_{i-1})$  that generate the Toeplitz flow $(X,T)$ and the jump function $p(x)$ for the speedup, find the tower permutations $\pi_i^{(k)}$ for the speedup for each KR-partition $\mathcal{P}(k)$.
        \item Define a substitution $\tau_{i+1}: \mathcal{A}_{i+1}\times I \to (\mathcal{A}_i\times I)^\ast$ by $\tau_{i+1}(A,j) = (b_1,j_1)(b_2,j_2)\cdots (b_n,j_n)$ where $\theta_{i+1}(A) = b_1, b_2\ldots b_n$, $j_1 = j$, and for all $i\geq 2$, $j_i = \pi_{b_{i-1}}^{(i)}(j_{i-1})$.
        \item Define a substitution $\psi_l$ that encodes the individual jumps within each tower orbit $(a,j)\in \mathcal{A}_l\times C$ into a new alphabet $\mathcal{C}$.
        \item Take a bi-infinite sequence $x\in X$ and decode it into into $\mathcal{A}_l\times C$-words and apply $\psi_l$; denote this new sequence as $x'$ and note that the shift space generated by $x'$ is conjugate to the speedup $(X,S)$. Note that if we choose $x\in X$ to be a Toeplitz sequence (under $T$), then the decomposition of $x$ into $\mathcal{A}_l\times C$ words can be found by extending the labeling rule defined in Definition~\ref{Definition: generalize tau} to $x$ where we assume $l_0 = 1$, $l_j= \pi_{i_j}(l_{j-1})$ if $j\geq 1$ and $l_j=\pi_{i_j}^{-1}(l_{j+1})$ if $j\leq -1$. Hence we may apply $\psi_l$ to this bi-infinite sequence and generate the same shift space conjugate to $(X,S)$.
        \item Suppose that some point $y\in X$ is Toeplitz under $S$. There is a unique decomposition of $y$ into $\psi_l(\mathcal{A}_l\times C)$-words for every $l\in \N$.
        We call this decomposition $y'$ and recall that the shift map on $y'$ acts in the same way that $S$ acts on $y$. Since the $0$th coordinate of $y'$ appears every $p$ positions (there exists $p\in \N$ since $y'$ is Toeplitz under the shift map), it follows that this symbol must appear every $p$ positions in every $\psi_l(\mathcal{A}_l\times C)$-word for every $l\in \N$.
        \item Thus, if one can show that no such $p$ can possibly exist (through checking larger and larger $l$ values), we obtain a contradiction and $(X,S)$ is not Toeplitz.
        \item If it follows that there is a $p\in \N$ such that the $0$th coordinate of $y'$ appears every $p$ positions in $y'$, then we set $p = p_1$ and we suppose there is some $p_2\in \N$ (necessarily a multiple of $p_1$) such that the word $y'_{[0,p_1-1]}$ appears every $p_2$ positions in every $\psi_l(\mathcal{A}_l\times C)$ word for every $l\in \N$ and continue recursively. If we can find some infinite sequence $(p_k)$ of integers, this will form our period structure for $y'$ under the shift map and we will have verified that $y$ is Toeplitz under $S$. If at any stage we cannot find a valid $p_k$, then the speedup will not be Toeplitz.
    \end{enumerate}

\newpage
\section{Sufficient Conditions for Obtaining Toeplitz Flows} \label{Sec: Sufficient Conditions}

{ Section~\ref{Sec: Determining if Toeplitz} characterized constant speedups of Toeplitz flows using the almost 1-1 extension structure over odometers and gave some insight into a strategy for determining whether a non-constant speedup of a Toeplitz flow is Toeplitz. In this section we focus on bounded speedups with non-constant jump functions. We investigate when non-constant speedups are again Toeplitz flows, when they are conjugate to a constant speedup, and what structural constraints the original flow imposes on the underlying odometer of the speedup. 
}

{ \subsection{Sufficient conditions for a speedup to be Toeplitz} We begin with sufficient conditions for a bounded speedup of a Toeplitz flow to be another Toeplitz flow. The first condition is stated in terms of the almost 1-1 extension structure; the second is a more general combinatorial translation in terms of the KR-partitions.}
Recall the following definition:
    for each $k\in \N$ large enough such that Remark~\ref{Remark: KR levels} holds and for each $j\in \{1,2,\ldots, c\}$, we let $o_i^k(j)$ denote the {\em height of orbit $j$ in tower $A_i^k$}; that is, $o_i^k(j)$ is precisely the number of floors in tower $A_i^k$ that are labeled $j$.

\begin{theorem}\label{Theorem: Sufficient Same Permutation and Orbits}
    { Let $(X,T)$ be a Toeplitz flow over $(G,\tau)$ with almost 1-1 factor map $h: X\to G$, and let $(X,S)$ be a minimal bounded speedup of $(X,T)$ with bounded jump function $p:X\to \N$ and orbit number $c$.}
    
    \begin{enumerate}
    {    \item If the level sets of $p$ are lifted from $G$ (i.e., $p=q\circ h$ for some locally constant $q: G\to \N$), then $(X,S)$ is a Toeplitz flow.}
        \item More generally, if for some sufficiently large $k$, the orbit permutation is the same on each tower of the KR-partition $P(k)$ and the orbit heights are the same in each tower (i.e., $o_i^k(j) = o_l^k(j)$ for each $j\in \{1,2,\ldots, c\}$ and all $i,l$), then $(X,S)$ is a Toeplitz flow.
    \end{enumerate}
    
\end{theorem}
\begin{proof}
    {We first prove (1). First note that the bounded speedup of a symbolic system is symbolic. As $p=q\circ h$ and $h\circ T = \tau \circ h$, the speedup on $X$ projects via $h$ to a speedup on $G$. To see this, define $U: G\to G$ by $U(g) =\tau^{q(g)}(g)$. Then $h(S(x)) = U(h(x))$ for all $x\in X$, so $h$ is a factor map from $(X,S)$ to $(G,U)$.
    
    Since $\tau$ is an isometry and $q$ is constant on clopen neighborhoods, the speedup $U$ preserves the metric structure of $G$. If $g$ and $g'$ are in the same clopen set on which $q$ is constant, then the distance between $U(g)$ and $U(g')$ is the same as the distance between $g$ and $g'$. It follows that $U$ is both equicontinuous and distal. Since $G$ is zero-dimensional, $(G,U)$ is a union of odometers. If $(G,U)$ is minimal, it is an odometer.
    
    Observe that the almost 1-1 property of $h$ is preserved, hence $(X,S)$ is an almost 1-1 extension of $(G,U)$. Hence, if $(X,S)$ is minimal, then $(G,U)$ is minimal and is thus an odometer. We conclude that $(X,S)$ is a minimal symbolic almost 1-1 extension of an odometer and hence it is Toeplitz.}

    For the proof of (2), first observe that if every tower has the same permutation for some fixed $k\in \N$, then we may denote $\pi_i^{(k)} = \pi$ for simplicity (when $k$ is understood). Also note that if $(X,S)$ is minimal, then $\pi$ must be a cyclic permutation on $\{1,2,\ldots, c\}$. If $\pi$ were not a cyclic permutation, then there would be some subcycle of length $\ell < c$ such that $\pi^\ell (1) = \pi \circ \cdots \circ \pi(1) = 1$. In this case we would have points in $X$ that do not enter into every orbit labeling, which contradicts the minimality of $(X,S)$.  We thus have that $\pi^c(j) = j$ for all $j\in \{1,2,\ldots, c\}$. By the construction of the KR-partitions, the orbit labeling, and the fact that $(X,T)$ is a Toeplitz flow, it will follow that for every $k'\geq k$, each tower in $\mathcal{P}(k')$ will have the same permutation.

    Now suppose for some $k\in\N$ we have that for every pair of towers $A_i^k$ and $A_l^k$ and every orbit label $j\in \{1,2,\ldots, c\}$ it holds that $o_i^k(j) = o_l^k(j)$. Because $(X,T)$ is Toeplitz, $\mathcal{P}(k+1)$ has precisely $N(k+1)$ towers each formed by stacking $p_{k+1}/p_k$ towers from level $\mathcal{P}(k)$ on top of each other (where $(p_k)$ is the period structure for $(X,T)$), and each tower in $\mathcal{P}(k)$ has the same permutation and the same orbit labeling, it follows that each tower in $\mathcal{P}(k+1)$ will also have the same orbit labeling. That is, for all $k'>k$, for any two towers $A_i^{k'}$ and $A_l^{k'}$ and any orbit label $j\in \{1,2,\ldots, c\}$ it holds that $o_i^{k'}(j)=o_l^{k'}(j)$. 

    Fix $k\in \N$ and let $x\in B(k)$, the base level of the KR-partition $\mathcal{P}(k)$. Because every tower in $\mathcal{P}(k)$ has the same cyclic permutation and the same number of floors labeled $j$ for all $j\in \{1,2,\ldots, c\}$, we have that $S^{o_i^k(1)+o_i^k(2)+\cdots + o_i^k(c)}(x) \in B(k)$. Let $m_k = o_i^k(1) + o_i^k(2) + \cdots + o_i^k(c)$. This implies that $S^{m_k\cdot n}(x)\in B(k)$ for all $x\in B(k)$ and  all $n\in \N$.

    Let $x = \bigcap_{k\in \N} B(k)$. 
    Fix $\eps > 0$; there exists some $K_\eps\in\N$ such that for all $k\geq K_\eps$ if $x,y\in B(k)$, then $d(x,y) < \eps$. For each $k$ large enough, we thus have that $d(S^{m_k\cdot n}(x), x)< \eps$. Hence $x$ is a Toeplitz point under $S$. Because $(X,S)$ is a minimal system that can be generated by a Toeplitz point, $(X,S)$ is thus a Toeplitz flow.
\end{proof}


{ \begin{remark} Observe that in Theorem~\ref{Theorem: Sufficient Same Permutation and Orbits}, part (1) is a special case of part (2): if $p=q\circ h$ for a locally constant function $q$, then $p$ acts the same way on each tower $P(k)$ for sufficiently large $k$ (since the towers refine the clopen sets on which $q$ is constant). In this case, the orbit permutations and the orbit heights are automatically the same across all towers at each sufficiently large level. We note that the two parts are not equivalent, as a jump function could produce the same orbit heights and same orbit permutations on each tower while differing on specific floors of each tower.\end{remark}

The conditions in Theorem~\ref{Theorem: Sufficient Same Permutation and Orbits}(2) require that both the orbit permutations agree across towers and the individual orbit heights agree. The following two questions ask whether each condition is independently necessary.}

\begin{question}
    If for all $\mathcal{P}(k)$, the towers have different orbit permutations, is it possible for the minimal speedup $(X,S)$ to be a Toeplitz flow? What if the original flow is generated by a constant length substitution?
\end{question}

\begin{question}
    If for all $\mathcal{P}(k)$, the towers have different orbit heights, is it possible for the minimal speedup $(X,S)$ to be a Toeplitz flow? What if the original flow is generated by a constant length substitution?
\end{question}  

{ We conjecture that the sufficient conditions of Theorem~\ref{Theorem: Sufficient Same Permutation and Orbits}(2) are in fact necessary, providing a complete characterization.}

\begin{conjecture}
    Suppose $(X,T)$ is a Toeplitz flow with period structure $(p_k)$ and suppose $(X,S)$ is a bounded speedup of $(X,T)$ with orbit number $c$. Then $(X,S)$ is a Toeplitz flow if and only if for all sufficiently large $k$, $\pi_i^{(k)}$ is the same  cyclic permutation on $\{1,2,\ldots, c\}$ for each tower of the KR-partition $\mathcal{P}(k)$ and the orbit heights are the same in each tower, (i.e., $o_i^k(m)=o_j^k(m)$ for each $m\in \{1,2,\ldots,c\}$ and all $i,j$).
\end{conjecture}

{\subsection{Coboundary conditions and conjugacy to constant speedups} We now turn to the question of when a Toeplitz speedup is conjugate to a constant speedup $(X,T^c)$. The answer is governed by a coboundary condition.}
Given a topological dynamical system $(X,T)$, { a} continuous function $g: X\to \R$ is a {\em $T$-coboundary} if there exists a continuous function $f: X\to \R$ such that { $g = f-f\circ T$}.
In \cite{AAO} it was shown that if $c$ is the orbit number for the speedup $(X,S)$ with jump function $p(x)$, then $p(x)-c$ is a $T$-coboundary.  Additionally, if $p(x)-c$ is an $S$-coboundary and $T^c$ is a minimal action on $X$, then $(X,S)$ is conjugate to $(X,T^c)$; we note that in general, $p(x)-c$ need not be an $S$-coboundary.

In the following proofs we use $p(x,n) = \sum_{i=0}^{n-1} { p(S^i(x))}$ to denote the unique natural number satisfying $S^n(x) = T^{p(x,n)}(x)$.

\begin{proposition}\label{Prop: Average Steps give S-coboundary}
    Let $(X,S)$ be the minimal bounded speedup of a minimal system $(X,T)$ with jump function $p(x)$ and orbit number $c$. There exists a level $\mathcal{P}(k)$ of the Kakutani-Rokhlin partition of $(X,T)$ { such that for all $x \in X$ and $x \in T^j(B_i^k)$ it holds that
    $$
        p(x,n) = cn
    $$
    for all $n\in\N$ such that $n$ is a return time of $x$ to its cylinder set $T^j(B_i^k)$ under $S$,}
    if and only if we can define a function $g \in C(X, \Z)$ such that
    $$
        p(x) = c + g(x) - g(S(x)) \text{\ for all $x\in X$.}
    $$
    In other words, if and only if $p(x)-c$ is an $S$-coboundary.
\end{proposition}
\begin{proof}
    Take $k$ such that $p(x)$ is constant on every cylinder set in $\mathcal{P}(k)$. As $S$ is minimal, for any KR-partition $\mathcal{P}(k)$ there exists $k'>k$ such that any $T$-tower of $\mathcal{P}(k')$ {  intersects} all towers of $\mathcal{P}(k)$ with every possible orbit labeling (see Lemma~\ref{Lemma: generalize tau is primitive}). We construct the function $g$ along the S-orbit of a base point of the partition. For $x_0\in B_1^{k'}$ we set $g(x_0) = 0$ and then for $x_1 = S(x_0)$ we have
    $$
        g(x_1) = c + g(x_0) - p(x_0).
    $$
    By { recursion} we can define $g$ for any $x$ in the S-orbit of $x_0$,
    $$
        g(x_n) = nc + g(x_0) - p(x_0,n) \text{\ for $x_n = S^n(x_0)$}.
    $$
    This construction gives every cylinder set $B^k_i(j)$ in $\mathcal{P}(k)$ the $g$-value of $x_l$ such that $x_l \in B^k_i(j)$; here $B^k_i(j)$ indicates the $j$-th floor of the $A_i^k$-tower in $\mathcal{P}(k)$. Thus for $p(x)-c$ to be an $S$-coboundary, any return to a previously visited cylinder set has to give the same value of $g$. Hence, for any $x$ in any base $B_i^k$ of the KR-partition $\mathcal{P}(k)$ and $n$ such that $S^n(x)$ returns to that base,
    \begin{align*}
        g(S^n(x)) = nc + g(x) - p(x,n) = g(x).
    \end{align*}
    Therefore the function $g$ is constant on every cylinder set of $\mathcal{P}(k)$.

    On the other hand, if there exists a continuous function $g: X \mapsto \Z$ and a level $k$ such that
    $$
        g(x) -g(S(x)) = p(x) - c  \text{\ for all $x\in X$ and $g(x) = g(y)$ for all $x,y \in B_i^k(j)$,}
    $$
    then for any return time $n$ of $x$ to $ B_i^k(j)$ under $S$
    \begin{align*}
        0   &= g(x) - g(S^n(x))\\
            &= \sum_{m=0}^{n-1} (g(S^m(x)) - g(S^{m+1}(x))\\
            &= \sum_{m=0}^{n-1} (p(S^m(x)) - c) = p(x,n) - nc.          
    \end{align*}
\end{proof}

\begin{eg}
    Consider the length $3$-Toeplitz shift on two symbols $\{a,b\}$ generated by the following substitution
    $$
    \chi: \begin{cases}
    a \to aab\\
    b \to abb
    \end{cases}
    $$
    and take any jump function with orbit number $c=2$ such that 
    \begin{itemize}
        \item the speedup $(X,S)$ is minimal,
        \item the orbit permutations are $\pi^{(k)}_a = \pi^{(k)}_b = (12)$ and
        \item $o_a^k(1) = o_b^k(1)$ (and therefore also $o_a^k(2) = o_b^k(2)$).
    \end{itemize}
    Then $(X,S)$ is conjugate to $(X,T^2)$ and is therefore Toeplitz.

    We may choose $k$ such that the KR-partition $\mathcal{P}(k)$ has a constant jump function on every floor and such that every tower in the KR-partition $\mathcal{P}(k+2)$ { intersects} every tower from $\mathcal{P}(k)$. To apply Proposition~\ref{Prop: Average Steps give S-coboundary}, we need to show that
    $$
        p(x,n) = 2n \text{ for all first return times under $S$ of $x$ to its cylinder set $B_i^k$.}
    $$    
    The jump function $p(x,n)$ counts the number of T-steps between $x \in B_i^k$ and its return to the same cylinder set. Then
    $$
        p(x,n) = mh^k  
    $$
    where $m$ is the number of T-towers traversed between those consecutive returns to $B_i^k$ and $h^k$ is the height of the towers in $\mathcal{P}(k)$.
    By the permutation $\pi^{(k)}$, every return under $S$ to the base passes through the same number of towers with orbit label $j$ and towers with orbit label $j'$; thus $m = \#(j) + \#(j') = 2\#(j)$, where $\#(j)$ denotes the number of towers traversed with orbit label $j$.
    For $x \in A_i^k$ on floor $d$ with orbit label $j$, we can count the number of $S$-steps until its return to the same floor
    \begin{align*}
        n &= o_i^k(j) + \# (i,j')o_i^k(j') + \# (i',j')o_{i'}^k(j') +\# (i',j)o_{i'}^k(j) \\
          &= (\# (i,j') + \# (i',j'))o_i^k(j')  +(\# (i',j) +1) o_{i}^k(j) \\
          &=\# (j')o_i^k(j')  +\# (j) o_{i}^k(j) = \#(j) h^k        
    \end{align*}
    where $\#(i,j)$ is the number of occurrences of towers $A_i^k$ with orbit label $j$ between $x$ and its return under $S$.
    Then it follows that
    $$
        p(x,n) = mh^k = 2\#(j) h^k = 2n.
    $$
\end{eg}

We now demonstrate that the only minimal speedups of the Toeplitz flow generated by $\chi$ with orbit number $c=2$ which are also Toeplitz are conjugate to $(X,T^2)$ and must arise as described in the above example.

\begin{eg}\label{Example: All speedups of ADL are same}
    Consider the Toeplitz shift $(X,T)$ on two symbols $\{a,b\}$ generated by 
    $$
    \chi: \begin{cases}
    a \to aab\\
    b \to abb.
    \end{cases}
    $$ Every minimal speedup with orbit number $c=2$ that is a Toeplitz flow is conjugate to $(X,T^2)$.
\end{eg}
\begin{proof}
    {\bf Case I:} Let $S$ be a minimal speedup with orbit number $c=2$. If $\pi_a^{(k)}=\pi_b^{(k)} = (12)$ for some $k\in \N$, then note that $\pi_a^{(l)} = \pi_b^{(l)} = (12)$ for all $l\geq k$.  Consider the substitution $\tau: \left( \{a,b\}\times \{1,2\}\right) \to \left( \{a,b\}\times \{1,2\}\right)^\ast$ as defined in Definition~\ref{Definition: tau} that models the speedup in terms of the towers and orbit paths traced.
    \begin{align*}
        \tau(a,1) &= (a,1)(a,2)(b,1)\\
        \tau(a,2) &= (a,2)(a,1)(b,2)\\
        \tau(b,1) &= (a,1)(b,2)(b,1)\\
        \tau(b,2) &= (a,2)(b,1)(b,2)
    \end{align*}
    Then $\tau$ is a primitive substitution that is conjugate to the substitution $\phi: \{A,B,C,D\}\to \{A,B,C,D\}^\ast$ by
    \begin{align*}
        \phi(A) & = ABC \quad & \text{ where }  & \quad A = (a,1)(a,2) \\
        \phi(B) & = CBC \quad &  & \quad B = (b,1)(a,2)\\
        \phi(C) & = ABD & & \quad C = (a,1)(b,2)\\
        \phi(D) & = CBD & & \quad D = (b,1)(b,2).
    \end{align*}
    Let $o_1(a)$ denote the number of jumps made by $S$ in $(a,1)$, $o_2(a)$ denote the number of jumps made by $S$ in $(a,2)$, and similarly for $o_1(b)$ and $o_2(b)$.  Note that $o_1(a) + o_2(a) = 3^k$ and $o_1(b)+o_2(b) = 3^k$ for some fixed $k\in \N$.

    We denote by $|A| = o_1(a)+o_2(a) = 3^k$ the number of jumps made by $S$ in $A = (a,1)(a,2)$. We similarly define $|B| = o_1(b) + o_2(a)$, $|C| = o_1(a) + o_2(b)$, and $|D| = 3^k$.

    Observe that if $(X,S)$ is Toeplitz, then there is some point $x\in X$ that is Toeplitz under $S$. Further, there is a unique decomposition of $x$ into its $\phi$-words. Using the minimality of $x$ and the definition of Toeplitz, there exist $l,j\in \N$ such that $S^{ln+j}(x)$ has a $\phi$-decomposition that begins with $A$ for all $n\in \N$.
    \begin{align*}
        \phi^2(A) &= ABCCBCABD \\
        \phi^2(B) &= ABDCBCABD \\
        \phi^2(C) &= ABCCBCCBD \\
        \phi^2(D) &= ABDCBCCBD
    \end{align*}
    Because $\phi^2(C)$ contains only one $A$, it follows that $l\geq |\phi^2(C)| = 8\cdot 3^k + o_1(a) + o_2(b)$. Similarly, we can conclude that $l\geq |\phi^2(D)| = 9\cdot 3^k$. Further observe that $|\phi^2(A)| = 9\cdot 3^k$ and $|\phi^2(B)| = 8\cdot 3^k + o_1(b) + o_2(a)$. Note that for all $m\in \N$, $\phi^m(A)$ and $\phi^m(C)$ differ in only one position and $\phi^m(C)$ and $\phi^m(D)$ differ in only one position.
    \begin{align*}
        \phi^3(A) & = \phi^2(A)\phi^2(B)\phi^2(C)  = ABCCBCABD ~ABDCBCABD~ ABCCBCCBD \\
        \phi^3(B) & = \phi^2(C)\phi^2(B)\phi^2(C)  = ABCCBCCBD ~ABDCBCABD~ ABCCBCCBD \\
        \phi^3(C) & = \phi^2(A)\phi^2(B)\phi^2(D)  = ABCCBCABD ~ABDCBCABD~ ABDCBCCBD \\
        \phi^3(D) & = \phi^2(C)\phi^2(B)\phi^2(D)  = ABCCBCCBD ~ABDCBCABD~ ABDCBCCBD
    \end{align*}
    Thus we return to $A$ every $l=9\cdot 3^k$ jumps if and only if $|\phi^2(A)| = |\phi^2(B)| = |\phi^2(C)| = |\phi^2(D)| = 9\cdot 3^k$, which occurs if and only if $|A| = |C|$ and $o_1(a)=o_1(b)$. Otherwise, $l\geq \max \{|\phi^2(B)\phi^2(D)|, |\phi^2(C)\phi^2(B)|, |\phi^2(A)\phi^2(B)|\} = \max \{17\cdot 3^k + o_1(b)+o_2(b), 18\cdot 3^k\}.$ If $l = 18\cdot 3^k$, we again deduce $o_1(a) = o_1(b)$; otherwise $l\geq \max\{|\phi^3(A)|, |\phi^3(B)|, |\phi^3(C)|, |\phi^3(D)|\}$. Inductively, we obtain that $l$ is a finite value if and only if $o_1(a) = o_1(b)$.

    We thus conclude that if $(X,S)$ is a Toeplitz flow with orbit number $c=2$ and $\pi_a^{(k)} = \pi_b^{(k)} = (12)$, then $(X,S)$ is topologically conjugate to $(X,T^2)$.

    {\bf Case II:} If $c = 2$ and $\pi_a^{(k)} = \pi_b^{(k)} = \text{id}$, then $(X,S)$ is not minimal.

    {\bf Case III:} Suppose we have that $\pi_a^{(1)} = (12)$ and $\pi_b^{(1)} = \text{id}$. Then $\pi_a^{(2l+1)} = (12)$ and $\pi_b^{(2l+1)} = \text{id}$ for all $l\in \N$. We again consider our substitution $\tau$ as defined in Definition~\ref{Definition: tau}, which is given by 
    \begin{align*}
        \tau(a,1) & = (a,1)(a,2)(b,1)(a,1)(a,2)(b,1)(a,1)(b,2)(b,2) \\
        \tau(a,2) & = (a,2)(a,1)(b,2)(a,2)(a,1)(b,2)(a,2)(b,1)(b,1) \\
        \tau(b,1) & = (a,1)(a,2)(b,1)(a,1)(b,2)(b,2)(a,2)(b,1)(b,1) \\
        \tau(b,2) & = (a,2)(a,1)(b,2)(a,2)(b,1)(b,1)(a,1)(b,2)(b,2).
    \end{align*}
    We obtain the conjugate substitution 
    \begin{align*}
        \phi(A) &= AABCDAE \quad &\quad \text{ where } \quad & A = (a,1)(a,2)(b,1)\\
        \phi(B) &= AABDBDBCD & & B = (a,1)(b,2)(b,2)(a,2)\\
        \phi(C) &= AABDBCD & & C =(a,1)(b,2)(a,2)\\
        \phi(D) &= AABDBCDAEAE & & D = (a,1)(b,2)(a,2)(b,1)(b,1)\\
        \phi(E) &= AABDBDBCBAEAE & & E =(a,1)(b,2)(b,2)(a,2)(b,1)(b,1).
    \end{align*}
    As in Case I, we get $|A| = 3^k + o_1(b)$, $|B| = 3^k + 2o_2(b)$, $|C| = 3^k + o_2(b)$, $|D| = 2\cdot 3^k + o_1(b)$, and $|E| = 3\cdot 3^k$ for some $k\in \N$.  Additionally, we have $|\phi(A)| = 13\cdot 3^k + o_1(b)$, $|\phi(B)| = 17\cdot 3^k + 2 o_2(b)$, $|\phi(C)| = 13\cdot 3^k + o_2(b)$, $|\phi(D)| = 22\cdot 3^k + o_1(b)$, and $|\phi(E)| = 27\cdot 3^k$.

    If there is a point $x\in X$ that is Toeplitz under $S$, then there exist $j,l\in \N$ such that $S^{ln+j}(x)$ has a $\phi$-decomposition that begins with $AAB$ for all $n\in \N$. However, we claim that there are no values for $o_1(b)$ and $o_2(b)$ such that this is possible.  Hence, it is not possible for $(X,S)$ to be Toeplitz when $c=2$ with $\pi_a^{(1)} = (12)$ and $\pi_b^{(1)}=\text{id}$.

    {\bf Case IV:} Suppose we have $\pi_a^{(1)} = \text{id}$ and $\pi_b^{(1)} = (12)$. Then $\pi_a^{(2l+1)} = \text{id}$ and $\pi_b^{(2l+1)} = (12)$ for all $l\in \N$. We again consider the substitution $\tau$ as in Definition~\ref{Definition: tau}, given by 
    \begin{align*}
        \tau(a,1) & = (a,1)(a,1)(b,1)(a,2)(a,2)(b,2)(a,1)(b,1)(b,2) \\
        \tau(a,2) & = (a,2)(a,2)(b,2)(a,1)(a,1)(b,1)(a,2)(b,2)(b,1) \\
        \tau(b,1) & = (a,1)(a,1)(b,1)(a,2)(b,2)(b,1)(a,2)(b,2)(b,1) \\
        \tau(b,2) & = (a,2)(a,2)(b,2)(a,1)(b,1)(b,2)(a,1)(b,1)(b,2).
    \end{align*}
    Similar to the above cases, we obtain the conjugate substitution 
    \begin{align*}
        \phi(A) &= ABC  \\
        \phi(B) &= ABCADAEAECC \\
        \phi(C) &= ABCADCC \\
        \phi(D) &= ABCADAECCADAECCADAEAECC \\
        \phi(E) &= ABCADAECCADAEAEDD.
    \end{align*}
    where \begin{align*}
         A& = (a,1)\\
        B& = (a,1)(b,1)(a,2)(a,2)(b,2)\\
        C& =(a,1)(b,1)(b,2)\\
        D& = (a,1)(b,1)(a,2)(b,2)(b,1)(a,2)(b,2)(b,1)(a,2)(a,2)(b,2)\\
        E& =(a,1)(b,1)(a,2)(b,2)(b,1)(a,2)(a,2)(b,2).
    \end{align*}
    By counting the number of jumps $S$ makes in each letter, we see that $|A| = o_1(a)$, $|B|= 2\cdot 3^k + o_2(a)$, $|C|= 3^k + o_1(a)$, $|D| = 4\cdot 3^k + 3o_2(a)$, and $|E| = 3\cdot 3^k + 2o_2(a)$ for some $k\in \N$. We further have $|\phi(A)| = 4\cdot 3^k + o_1(a)$, $|\phi(B)| = 22\cdot 3^k + o_2(a)$, $|\phi(C)| = 13\cdot 3^k + o_1(a)$, $|\phi(D)| = 48\cdot 3^k + 3o_2(a)$, and $|\phi(E)| = 35\cdot 3^k + 2o_2(a)$.

    If there is a point $x\in X$ that is Toeplitz under $S$, then there exist $j,l\in \N$ such that $S^{ln+j}(x)$ has a $\phi$-decomposition that begins with $ABC$ for all $n\in \N$. Similar to Case III, due to the formulas for the number of $S$-jumps in each letter, there are no possible values for $o_1(a)$ and $o_2(a)$ such that any such $l\in \N$ can exist. We thus conclude that it is not possible for $(X,S)$ to be Toeplitz when $c=2$, $\pi_a^{(1)}=\text{id}$, and $\pi_b^{(1)}=(12)$.

    By considering all cases, we conclude that the only minimal speedups $(X,S)$ of $(X,T)$  (where (X,T) is generated by $\chi$) with orbit number $c=2$ that are Toeplitz occur precisely when $\pi_a = \pi_b = (12)$ and $o_1(a) =o_1(b)$. Therefore, every Toeplitz speedup of $(X,T)$ with $c=2$ is conjugate to $(X,T^2)$.
\end{proof}

{ We now demonstrate that when the sufficient conditions outlined in Theorem~\ref{Theorem: Sufficient Same Permutation and Orbits} are satisfied, then $p(x)-c$ will be an $S$-coboundary. }

\begin{theorem}\label{Theorem: Sufficient for S-coboundary}
	Suppose that $(X,T)$ is a Toeplitz flow and $(X,S)$ is a minimal speedup of $(X,T)$.  If there exists some $K\in \N$ such that for the KR-partition $\mathcal{P}(K)$ we have that every tower has the same cyclic orbit permutation and each orbit has the same `height' in each tower (i.e., for every $1\leq j \leq c$, $o_i^K(j) = o_{i'}^K(j)$ for all towers $i,i'$), then $p(x) - c$ is an $S$-coboundary.
\end{theorem}

\begin{proof}
    { Let $K$ be the level for which the assumptions of the theorem hold. Then the same follows for every level $k > K$, as $\pi^{(K+1)}_{i}$ is a length $\lvert \theta_{(K+1)}\rvert$-concatenation of the permutation $\pi^{(K)}$ independent of $i$ and therefore $\pi^{(K+1)}_{i} = \pi^{(K+1)}_{i'}$ for all $i, i' \in \mathcal{A}_{(K+1)}$. Similarly $o_i^{(K+1)}(j)$ is the sum of $\lvert \theta_{(K+1)}\rvert$ values of $o_i^K(j')$ for some labels $j'$ dependent on $\pi^{(K)}$ but not $i$.
    Thus let $k>K$} be such that the jump function $p(x)$ is constant on every cylinder set of $\mathcal{P}(k)$. For an arbitrary $x\in B_i^k$, take a return time $n$ such that $S^n(x) = T^{p(x,n)}(x) \in B_i^k$. Therefore we know $h^k =p_k$ divides $p(x,n)$ and by the cyclic permutation $\pi^{(k)}$, it takes a multiple of $c$ steps to return to the same label. Thus $p(x,n) = cmh^k$ where $m$ is the number of $T$-towers with label $j$ between $x$ and $T^{p(x,n)}(x)$. To compute the number of steps under $S$, we take sums over all labels $j$ and towers $i$ between $x$ and $S^n(x)$, recall that $\#(i,j)$ is the number of occurrences of tower $A_i^k$ with label $j$,
    \begin{align*}
        n &= \sum_j\sum_i \#(i,j)o_i^k(j) = \sum_j  o_{i_o}^k(j) \sum_i \#(i,j)\\
            &= \sum_j  o_{i_o}^k(j) \#(j)= m\sum_j  o_{i_o}^k(j)\\
            &= mh^k 
    \end{align*}
    with $o_i^k(j) = o_{i'}^k(j)$ for all towers $i,i'$ and by the permutation every label has the same number of occurrences $\#(j)=\#(j')$ in the $cm$ towers between. Then $p(x,n) = cmh^k = cn$ and by Proposition \ref{Prop: Average Steps give S-coboundary} $p(x) - c$ is an $S$-coboundary.
\end{proof}

\begin{corollary}\label{Corollary: Sufficient Conjugacy to T^c}
    Suppose that $(X,T)$ is a Toeplitz flow and $(X,S)$ is a minimal speedup of $(X,T)$ such that $\gcd(c,p_k) = 1$ for all $k$ and the assumptions of the previous theorem hold. Then $(X,S)$ is conjugate to $(X,T^c)$.
\end{corollary}
\begin{proof}
Under the condition $\gcd(c,p_k) = 1$ for all $k$ we know that $(X,T^c)$ is a Toeplitz flow, see Theorem \ref{Theorem: constant speedup minimality characterization}. As $p(x) - c$ is an $S$-coboundary, Proposition 2.6 of \cite{AAO} proves the conjugacy.
\end{proof}

{ This motivates the following conjecture, which asks whether the coboundary condition is forced by the assumption
that the speedup is Toeplitz.}

\begin{conjecture}
	If $(X,T)$ is a Toeplitz flow and $(X,S)$ is a minimal speedup of $(X,T)$ that is also a Toeplitz flow, then $p(x)-c$ is an $S$-coboundary. Further, if $\gcd(c,p_k) = 1$ for all $k$, then $(X,S)$ is topologically conjugate to $(X,T^c)$.
\end{conjecture}

{
\subsection{Factors of constantly sped-up systems}
As the previous section demonstrated that many speedups of Toeplitz flows are conjugate to constant speedups, we take a moment to further analyze Toeplitz flows under constant speedups to gain a better understanding of the underlying structure. We focus on the length $3$-Toeplitz system in Example~\ref{Example: All speedups of ADL are same} to look for factors of constantly sped-up systems. We use Theorem~\ref{Theorem: constant speedup substitution} to define the corresponding speedup substitutions.}

    
\begin{eg}\label{Example: const 2 speedup of 3-Teoplitz system}
    Let $(X_\chi, T)$ be a Toeplitz subshift generated by $\chi: a\mapsto aab, \ b\mapsto abb$. The substitution $\tau$ of the constant 2-speedup will be on the alphabet $\mathcal{B} =\{A,B,C,D\}$ where we sort the 2-words lexicographically. Then $\tau$ is defined in the following way.
    \begin{align*}
            \tau(A) &= \chi(aa) = aabaab = ACB \\
            \tau(B) &= \chi(ab) = aababb = ACD \\
            \tau(C) &= \chi(ba) = abbaab = BCB \\
            \tau(D) &= \chi(bb) = abbabb = BCD
    \end{align*}
    We see that $\tau^2$ is a proper, primitive substitution of constant length $3^2$.
        \begin{align*}
            \tau^2(A) &= ACB \ BCB\ ACD\\
            \tau^2(B) &= ACB \ BCB\ BCD\\
            \tau^2(C) &= ACD \ BCB\ ACD\\
            \tau^2(D) &= ACD \ BCB\ BCD
    \end{align*}
        
    To find a factor map $\mathcal{F}$ from $(X_\tau, T^2)$ to $(X_\chi, T)$ we use techniques from~\cite{CovenQuasYassawi2016}. First we scale both substitutions to the same length and analyze which positions of the substitution are not coincidences.
    For $\tau$ there are two such positions in $w_0 w_1 \cdots w_8$: at position $2$ there is a choice between $B$ and $D$ and at position $6$ there is a choice between between $A$ and $B$.
    For the original $\chi$ we see that the only difference is at position $4$
    \begin{align*}
            \chi^2(a) &= aab\ aab\ abb\\
            \chi^2(b) &= aab\ abb\ abb.
    \end{align*}
    For a factor map $\mathcal{F}$ to exist, we need to map non-singular fibers of $X_\tau$ to all non-singular fibers of $X_\chi$. All other non-singular fibers must collapse to singular fibers. First we shift $\chi^2$ by 7 such that the difference is at the same position as a difference in $\tau^2$ and compare:
    \begin{align*}
            \chi'^2(a) =&\ b\hspace{3pt}  b\hspace{3pt}  a\hspace{9pt} a\hspace{3pt}  b\hspace{3pt}  a\hspace{9pt} a\hspace{3pt}  b\hspace{3pt}  a\\
            \chi'^2(b) =&\ b\hspace{3pt}  b\hspace{3pt}  a\hspace{9pt} a\hspace{3pt}  b\hspace{3pt}  a\hspace{9pt} b\hspace{3pt}  b\hspace{3pt}  a \\
            \tau^2(A) =&\ ACB \ BCB\ ACD\\
            \tau^2(B) =&\ ACB \ BCB\ BCD\\
            \tau^2(C) =&\ ACD \ BCB\ ACD\\
            \tau^2(D) =&\ ACD \ BCB\ BCD.
    \end{align*}
    We define the map $\mathcal{F}$ such that 
    \begin{equation*}
        \mathcal{F}(A) = b, \ 
        \mathcal{F}(B) = a, \
        \mathcal{F}(C) = b, \
        \mathcal{F}(D) = a,    
    \end{equation*}
    the non-singular fiber of the pair $\{B,D\}$ collapses to a singular fiber, and the pair $\{A,B\}$ stays non-singular. The choice of which fiber collapses was arbitrary, as there exists another factor map where $\{A,B\}$ collapses by shifting $\chi^2$ by 2.
    Constant speedups with an even orbit number have no non-singular fiber at the middle position, thus these cases always need to shift the substitution to get overlapping differences to the original system $X_\chi$.
    \end{eg}
    Factor maps from other constant speedups of the same Toeplitz system onto the original $X_\chi$ can be found in a similar way. We now take a few moments to comment on several examples to highlight the connection between specific constant speedups of the same Toeplitz flow.
    
    \begin{eg} Let $(X_\chi, T)$ be a Toeplitz subshift generated by $\chi: a\mapsto aab, \ b\mapsto abb$.
    \begin{enumerate}
        \item Letting $c = 5$, there are 10 words of length 5 in the language of $X_\chi$. The proper, primitive substitution of the speedup $(X_\chi,T)$ is of length 27.

    \begin{align*}
        A&\mapsto AHC EHI BGD \ BHC EHI FGD \ AHD EGI BGJ\\
        B&\mapsto AHC EHI BGD \ BHC EHI FGD \ AHD EGI FGJ\\
        C&\mapsto AHC EHI BGJ \ BHC EGI FGD \ AHD EGI FGJ\\
        D&\mapsto AHC EHI BGJ \ BHC EGI FGD \ BHD EGI FGJ\\
        E&\mapsto AHC EHI BGJ \ BHC EHI FGD \ AHD EGI BGJ\\
        F&\mapsto AHC EHI BGJ \ BHC EHI FGD \ AHD EGI FGJ\\
        G&\mapsto AHD EHI BGD \ BHC EGI FGD \ BHD EGI BGJ\\
        H&\mapsto AHD EHI BGD \ BHC EHI FGD \ BHD EGI BGJ\\
        I&\mapsto AHD EHI BGJ \ BHC EGI FGD \ AHD EGI FGJ\\
        J&\mapsto AHD EHI BGJ \ BHC EGI FGD \ BHD EGI FGJ\\
        \intertext{We see non-singular fibers at positions $2$ with $\{C,D\}$, at $8$ with $\{D,J\}$, at $13$ with $\{G,H\}$, at $18$ with $\{A,B\}$, and at $24$ with $\{B,F\}$. The original substitution of length 27 is}
            a&\mapsto \ 
            a\hspace{3pt}  a\hspace{3pt}  b\hspace{3pt} a\hspace{3pt}  a\hspace{3pt}  b\hspace{3pt} a\hspace{3pt}  b\hspace{3pt}  b\hspace{9pt} 
            a\hspace{3pt}  a\hspace{3pt}  b\hspace{3pt} a\hspace{3pt}  a\hspace{3pt}  b\hspace{3pt} a\hspace{3pt}  b\hspace{3pt}  b\hspace{9pt}
            a\hspace{3pt}  a\hspace{3pt}  b\hspace{3pt} a\hspace{3pt}  b\hspace{3pt}  b\hspace{3pt} a\hspace{3pt}  b\hspace{3pt}  b \\  %
            b&\mapsto \ 
            a\hspace{3pt}  a\hspace{3pt}  b\hspace{3pt} a\hspace{3pt}  a\hspace{3pt}  b\hspace{3pt} a\hspace{3pt}  b\hspace{3pt}  b\hspace{9pt} 
            a\hspace{3pt}  a\hspace{3pt}  b\hspace{3pt} a\hspace{3pt}  b\hspace{3pt}  b\hspace{3pt} a\hspace{3pt}  b\hspace{3pt}  b\hspace{9pt}
            a\hspace{3pt}  a\hspace{3pt}  b\hspace{3pt} a\hspace{3pt}  b\hspace{3pt}  b\hspace{3pt} a\hspace{3pt}  b\hspace{3pt}  b.
    \end{align*}
        To collapse all but the middle fiber we map $\{A, B, E, F, H\}$ to $a$ and $\{C, D, G,I,J\}$ to $b$. Hence we see that the original Toeplitz flow is a factor of the Toeplitz flow $(X_\chi, T^5)$.
        \item If we take $c = 7$, then there are 14 words of length 7 in the language of $X_\chi$. The seven non-singular fibers are at positions 1 with $\{L, M\}$, at 5 with $\{H, I\}$, at 9 with $\{A, B\}$, at 13 with $\{C, E\}$, at 17 with $\{K, I\}$, at 21 with $\{B, G\}$, and at 25 with $\{ D, N\}$. Thus, the factor maps $\{A, B, C, F, G, L, M\}$ to $a$ and $\{D, E, H, I, J, K, N\}$ to $b$. Again, the original Toeplitz flow is a factor of the Toeplitz flow $(X_\chi, T^7)$.
        \item For the constant speedup by $c=4$, we take the alphabet $\{0, 1, \ldots, 7 \}$ of size 8. We find 4 non-singular fibers at positions 1 with $\{2,7\}$, at 3 with $\{0,3\}$, at 5 with $\{4, 5\}$ and at 7 with $\{1,2\}$. 

        To map this onto the constant $2$-speedup (see Example \ref{Example: const 2 speedup of 3-Teoplitz system}), we shift the $2$-speedup once to have the differences at the same positions.
    \begin{align*}
        0&\mapsto 0\hspace{3pt}2\hspace{3pt}4 \hspace{9pt} 0\hspace{3pt}6\hspace{3pt}5 \hspace{9pt} 3\hspace{3pt}1\hspace{3pt}5\\
        1&\mapsto 0\hspace{3pt}2\hspace{3pt}4 \hspace{9pt} 3\hspace{3pt}6\hspace{3pt}4 \hspace{9pt} 3\hspace{3pt}1\hspace{3pt}5\\
        2&\mapsto 0\hspace{3pt}2\hspace{3pt}4 \hspace{9pt} 3\hspace{3pt}6\hspace{3pt}4 \hspace{9pt} 3\hspace{3pt}2\hspace{3pt}5\\
        3&\mapsto 0\hspace{3pt}2\hspace{3pt}4 \hspace{9pt} 3\hspace{3pt}6\hspace{3pt}5 \hspace{9pt} 3\hspace{3pt}1\hspace{3pt}5\\ 
        4&\mapsto 0\hspace{3pt}7\hspace{3pt}4 \hspace{9pt} 0\hspace{3pt}6\hspace{3pt}4 \hspace{9pt} 3\hspace{3pt}2\hspace{3pt}5\\ 
        5&\mapsto 0\hspace{3pt}7\hspace{3pt}4 \hspace{9pt} 0\hspace{3pt}6\hspace{3pt}5 \hspace{9pt} 3\hspace{3pt}2\hspace{3pt}5\\ 
        6&\mapsto 0\hspace{3pt}7\hspace{3pt}4 \hspace{9pt} 3\hspace{3pt}6\hspace{3pt}4 \hspace{9pt} 3\hspace{3pt}1\hspace{3pt}5\\ 
        7&\mapsto 0\hspace{3pt}7\hspace{3pt}4 \hspace{9pt} 3\hspace{3pt}6\hspace{3pt}4 \hspace{9pt} 3\hspace{3pt}2\hspace{3pt}5\\
        A&\mapsto CBB \ CBA \ CDA\\
        B&\mapsto CBB \ CBB \ CDA\\
        C&\mapsto CDB \ CBA \ CDA\\
        D&\mapsto CDB \ CBB \ CDA
    \end{align*}
    Thus, we see that mapping $\{5\}$ to A, $\{4,6,7\}$ to B, $\{0,3\}$ to C and $\{1,2\}$ to D gives us a factor map from the constant $4$-speedup to the constant $2$-speedup. That is, $(X_\chi, T^2)$ is a factor of $(X_\chi, T^4)$.
    \item For the constant speedup by $c=8$ or $c=10$, a similar procedure shows that the system sped-up by any divisor of $c$ is a factor. The constant $8$-speedup has the constant $4$-speedup as a factor and therefore also the constant $2$-speedup and the original system as factors. The constant $10$-speedup has the constant $2$-, the constant $5$-speedup, and the original system as factors.
    \end{enumerate}
    \end{eg}

    We believe that the relationships highlighted in the previous example hold for any constant length left proper substitution $\theta$, any $c>1$ with $\gcd(|\theta|,c)=1$, and any factor $l$ of $c$.

    \begin{conjecture}
        Given any constant length, left proper substitution $\theta$ and $c>1$ such that $\gcd(c,|\theta|)=1$, if $l$ is a factor of $c$ with $1\leq l< c$, then the Toeplitz flow $(X_\theta, T^l)$ is a factor of $(X_\theta, T^c)$.
    \end{conjecture}

{\subsection{Toeplitz speedups that are not conjugate to constant speedups} We previously showed that many Toeplitz speedups are conjugate to constant speedups.}
We now demonstrate that there exist minimal speedups $(X,S)$ of Toeplitz flows $(X,T)$ with orbit number $c$ such that $(X,S)$ is Toeplitz and not conjugate to $(X,T^c)$. By construction, the speedup $(X,S)$ in the next example will have an $S$-coboundary for $p(x)-c$ even though $(c,p_k) > 1$ for every $k$.
\begin{eg}\label{Example: Not conjugate to constant}
    Let $\mathcal{A} = \mathcal{A}_0 = \{0,1\}$ and for each $i\in \N$, define $\mathcal{A}_i = \{a,b\}$. Define $\theta_1:\mathcal{A}_1\to \mathcal{A}$ by $\theta_1(a) = 1001$ and $\theta_1(b) = 1010$. For all $i\geq 2$, define $\theta_i(a) = aab$ and $\theta_i(b) = abb$. Then the S-adic shift generated by this sequence of substitutions generates a Toeplitz flow $(X,T)$ with period structure $(p_k) = (4\cdot 3^k)$.

    Let $A = [\theta_1(\theta_2(a))] = \{x\mid x_{[0,11]}=1001~1001~1010\}$ and $B = [\theta_1(\theta_2(b))] = \{x\mid x_{[0,11]}=1001~1010~1010\}$ and define $p:X\to \Z^+$ by
    $$p(x) = \begin{cases} 2 & \text{if } x\in A \cup T^4(A) \cup T^8(A)\cup B\cup T^4(B) \cup T^8(B) \\
    2 & \text{if } x\in T(A)\cup T^5(A)\cup T^9(A)\cup T(B)\cup T^5(B)\cup T^9(B)\\
    3 & \text{if } x\in T^2(A)\cup T^6(A)\cup T^{10}(A)\cup T^2(B)\cup T^6(B)\cup T^{10}(B)\\
    1 & \text{if } x\in T^3(A)\cup T^7(A)\cup T^{11}(A)\cup T^3(B)\cup T^7(B)\cup T^{11}(B).\end{cases}$$

    With this given jump function $p$, $(X,S)$ is a minimal bounded speedup of $(X,T)$ where $S(x) = T^{p(x)}(x)$. In this case, the orbit number is $c=2$.
    It is easy to check that for all $k\in \N$, the orbit permutations are $\pi_a^{(k)}=\pi_b^{(k)} = (12)$ and the orbit heights are such that $o_a^k(1) = o_a^k(2) = o_b^k(1) = o_b^k(2)$. By Theorem~\ref{Theorem: Sufficient Same Permutation and Orbits}, it follows that $(X,S)$ is a Toeplitz flow.

    We can represent $(X,S)$ as an S-adic shift generated by the sequence of alphabets $\{\mathcal{B}_i\}_{i\geq 0}$ and substitutions $(\phi_i\}_{i\geq 1}$ given by $\mathcal{B}_0 = \{0,1,2,\ldots, 6\}$, $\mathcal{B}_i = \{a,b,c,d\}$ for all $i\geq 1$, $\phi_1(a) = 012304230156$, $\phi_1(b) = 015604230156$, $\phi_1(c) = 012304230456$, $\phi_1(d) = 015604230456$, and for all $i\geq 2$ define $\phi_i(a) = abccbcabd$, $\phi_i(b) = abdcbcabd$, $\phi_i(c) = abccbccbd$, and $\phi_i(d)= abdcbccbd$. 
    
    Observe that $(X,S)$ has period structure $(q_k) = (12\cdot 9^k)$. We conclude that $(X,S)$ and $(X,T)$ have the same underlying odometer.  Further note that although $p(x) -c$ is an $S$-coboundary by Theorem \ref{Theorem: Sufficient for S-coboundary}, $(X,S)$ is not topologically conjugate to $(X,T^c)=(X,T^2)$, since $(X,T^2)$ is not minimal.
\end{eg}




Recall that there exists a speedup of the odometer $(X_{\alpha},T)$ with $\alpha = ( \alpha_1,\alpha_2, \ldots )$ that has orbit number $c$ if and only if there is some $N\in \N$ such that $\gcd(c,\alpha_i) = 1$ for all $i\geq N$. We now demonstrate how one can generalize Theorem~\ref{Theorem: Odometer Speedup Characterization} to the Toeplitz flow setting.

\begin{theorem}\label{Theorem: Generalize Odometers To Toeplitz}
    If $(X,T)$ is a Toeplitz flow with period structure $(p_{k})$ and $c\geq 1$ is an integer such that for some $N\in \N$, $\gcd\left(c, \frac{p_{k+1}}{p_k}\right) = 1$ for all $k\geq N$, then there is a minimal bounded speedup $(X,S)$ of $(X,T)$ with orbit number $c$ such that $(X,S)$ is a Toeplitz flow with the same underlying odometer as $(X,T)$. Further, $(X,S)$ and $(X,T)$ are not conjugate.
\end{theorem}
\begin{proof}
    Let $(X,T)$ be a Toeplitz flow with period structure $(p_k)$ and underlying odometer given by $\alpha = (\alpha_1,\alpha_2,\ldots)$ where $\alpha_1 = p_1$, and  $\alpha_i = \frac{p_i}{p_{i-1}}$ for all $i\geq 2$. Suppose there exist $c, N\in \N$ such that for all $i\geq N$, $\gcd(c,\alpha_i) = 1$. By { Theorem \ref{Theorem: Odometer Speedup Characterization}, there exists some minimal bounded speedup of the odometer $(X_\alpha, T_\alpha)$ with orbit number $c$; further, by Lemma \ref{Lemma: cyclic permutations}}, there exists $K$ such that for all $k\geq K$ $\pi^{(k)}$ is a cyclic permutation on $\{1,2,\ldots, c\}$ for $(X_\alpha, T_\alpha)$. Let $M = \max(N,K)$. Define the jump function $p(x)$ on the towers of the KR-partition for $(X,T)$ at level $M$ so that each tower has the same orbit labellings as the single tower at level $M$ for the KR-partition for $(X_\alpha, T_\alpha)$ with orbit number $c$. Therefore, for $(X,T)$ with speedup $p(x)$, every tower has the same orbit labeling permutation $\pi^{(M)}$ which is a cyclic permutation on $\{1,2,\ldots, c\}$; additionally, the heights for each orbit agree in every tower.

    We first show that $(X,T^{p(x)})$ is minimal. Note that for every level $k$ in the sequence of KR-partitions for $(X,T)$, every tower can be interpreted as a stack of towers from the $(k-1)$st level such that the base tower is the same tower from the $(k-1)$st level, which we will call the {\em primary tower at level $k-1$}. Without loss of generality, suppose that every tower at level $k$ contains every tower from level $k-1$ (we may always choose a subsequence of our KR-partitions so this is true). Since $\pi^{(k)}$ is a permutation of $\{1,2,\ldots, c\}$ for all $k\geq M$ and every tower has the same permutation, if one starts on the base floor of any tower at level $k\geq M$, one will traverse through every orbit labelling of every tower from level $k-1$ after traversing through $c$ towers. As this is true for all $k\geq M$, we conclude that our speedup will be minimal.

    Lastly, because every tower in the KR-partition for $(X,T)$ with $k$ large enough has every tower with the same orbit permutation and every orbit has the same height in each tower, by Theorem~\ref{Theorem: Sufficient Same Permutation and Orbits}, we conclude that $(X,S) = (X,T^{p(x)})$ is a Toeplitz flow. By Theorem~\ref{Theorem: symbolic speedups not conjugate}, we conclude that $(X,S)$ and $(X,T)$ are not conjugate.
\end{proof}

\begin{remark}
    Observe that it immediately follows from Theorem~\ref{Theorem: Generalize Odometers To Toeplitz} that every Toeplitz flow $(X,T)$ with underlying odometer $(X_\alpha, T_\alpha)$ has a minimal speedup $(X,S)$ which is a Toeplitz flow if there exists some prime $p$ such that $p$ only divides finitely many $\alpha_i$.
    Thus, given a Toeplitz flow $(X, T)$ with underlying odometer given by $\alpha$, where $\alpha = (\alpha_1,\alpha_2,\ldots)$ does not have every natural number as a divisor, it will follow that there is some minimal speedup that generates a Toeplitz flow with the same underlying odometer. In particular, a Toeplitz flow with underlying odometer given by $\alpha  = (2,3,5,7,\ldots)$ will have a minimal speedup which is a Toeplitz flow for every orbit number $c\in \N$.
\end{remark}

{ The following conjecture gives a potential restriction to the existence of Toeplitz speedups with a given orbit number. Note that Theorem~\ref{Theorem: no speedup with c=2 is Toeplitz} proves the special case where $\theta(a) = ab$, $\theta(b) = aa$, and $c=2 = |\theta|$, and Question~\ref{Question: constant length substitutions no orbit number constant length} asks a similar question for general constant-length substitutions.}

\begin{conjecture}
    If $\gcd\left(c,\frac{p_k}{p_{k-1}}\right)>1$ for infinitely many $k$, then $(X,T)$ does not have a minimal speedup with orbit number $c$ that is a Toeplitz flow. Hence, if $\theta$ is a constant length left proper substitution, then the Toeplitz flow $(X_\theta,T)$ does not have a minimal speedup with orbit number $c = |\theta|$ that is also a Toeplitz flow. 
\end{conjecture}

{ We conclude with the natural question about whether a minimal bounded Toeplitz speedup must have the same underlying odometer as the original flow.}

\begin{question}
    Is it possible to speedup a Toeplitz flow in such a way to obtain a Toeplitz flow with a different underlying odometer? { Given any minimal bounded speedup of a Toeplitz flow, is the maximal equicontinuous factor always the original underlying odometer?}
\end{question}

\medskip
\textbf{Acknowledgements.}
SR would like to thank the OeAD Marietta Blau-Scholarship for funding her research stay at Furman University. 
LA gratefully acknowledges financial support for this project from the Henry Keith and Ellen Hard Townes endowed professorship at Furman University and from the Fulbright U.S. Scholar Program.
Additionally, the authors would like to thank Sarah Frick for helpful conversations in the early stages of this project and Reem Yassawi for her expertise in finding factor maps of constant length substitutions.
{ The authors are grateful to the reviewers for their valuable comments and suggestions.}
\bibliographystyle{plain}
\bibliography{BIBspeedups}

\end{document}